\documentclass[10pt]{article}
\usepackage{amssymb}
\usepackage{amsmath}
\usepackage{amsthm}
\usepackage{color}
\usepackage{colortbl}
\usepackage{youngtab}

\newcommand{\col}{\mathrm{col}} 
\newcommand{\row}{\mathrm{row}} 
\newcommand{\pivL}{\mathrm{piv}_l} 
\newcommand{\pivR}{\mathrm{piv}_r} 

\newcommand{\rev}{\mathrm{rev}} 
 
\newcommand{\ins}{\mathrm{ins}} 
\newcommand{\rec}{\mathrm{rec}} 

\newcommand{\leftexp}[2]{{\vphantom{#2}}^{#1}{#2}}
\newcommand{\shape}{\mathrm{shape}} 

\newcommand{\DL}[1]{\mathbf{D}_{#1}}
\newcommand{\ds}[1]{\mathbf{d}_{#1}}
\newcommand{\IL}[1]{\mathbf{I}_{#1}}
\newcommand{\is}[1]{\mathbf{i}_{#1}}
\theoremstyle{definition}
\newtheorem*{definition}{Definition}
\newtheorem*{definitions}{Definitions}
\newtheorem{theorem}{Theorem}
\newtheorem{lemma}{Lemma}
\newtheorem{remark}{Remark}
\newtheorem*{notation}{Notation}

\begin{document}\title{Modified Growth Diagrams, Permutation Pivots, and the BWX Map $\phi^*$}
\author{Jonathan Bloom\\Dartmouth College \and Dan Saracino\\Colgate University}
\date{}
\maketitle
\begin{abstract}

In their paper on Wilf-equivalence for singleton classes, Backelin, West, and Xin introduced a transformation $\phi^*$, defined by an iterative process and operating on (all) full rook placements on Ferrers boards. Bousquet-M$\acute{\textrm{e}}$lou and Steingr$\acute{\textrm{\i}}$msson proved the analogue of the main result of Backelin, West, and Xin in the context of involutions, and in so doing they needed to prove that $\phi^*$ commutes with the operation of taking inverses. The proof of this commutation result was long and difficult, and Bousquet-M$\acute{\textrm{e}}$lou and Steingr$\acute{\textrm{\i}}$msson asked if $\phi^*$ might be reformulated in such a way as to make this result obvious.  In the present paper we provide such a reformulation of $\phi^*$, by modifying the growth diagram algorithm of Fomin. This also answers a question of Krattenthaler, who noted that a bijection defined by the unmodified Fomin algorithm obviously commutes with inverses, and asked what the connection is between this bijection and $\phi^*$.

\end{abstract}
\noindent \textbf{1. Introduction}
\vspace{.25in}

For any permutation $\tau=\tau_1\tau_2\ldots \tau_r,$ let $S_n(\tau)$ denote the set of permutations in $S_n$ that avoid $\tau$, in the sense that they have no subsequence order-isomorphic to $\tau.$

In their paper \emph{Wilf-equivalence for singleton classes} [1], Backelin, West, and Xin prove an important general result about permutations avoiding a single pattern:  If $k, \ell \geq 1$ and $\rho$ is a permutation of $\{k+1,\ldots, k+\ell\}$, then for every $n\geq k+\ell$ we have $|S_n(12\ldots k\rho)|=|S_n(k\ldots 1\rho)|.$  The key tool in the proof is a map $\phi^*$, which operates on a permutation $\sigma$ as follows: Order the $k\ldots1$-patterns $\sigma_{i_1}\ldots \sigma_{i_k}$in $\sigma$ lexicographically (according to the $\sigma_j$'s, not the $j$'s) and let $\phi(\sigma)$ be obtained from $\sigma$ by taking the smallest $\sigma_{i_1}\ldots \sigma_{i_k}$ and placing in positions $i_1,\ldots, i_k$ the values $\sigma_{i_2},\ldots, \sigma_{i_k}, \sigma_{i_1}$, respectively, and leaving all other entries of $\sigma$ fixed. Let $\phi^*(\sigma)$ be obtained by applying $\phi$ repeatedly until no $k\ldots1$-patterns remain.  It is shown that $\phi^*$ induces a bijection from $S_n(k-1\ldots1k)$ onto $S_n(k\ldots 1)$, and that from this bijection one can obtain an important ingredient of the proof, namely a bijection from $S_n(1\ldots k)$ onto $S_n(k\ldots1).$

In [3], Bousquet-M$\acute{\textrm{e}}$lou and Steingr$\acute{\textrm{\i}}$msson prove the analogue of the Backelin-West-Xin Theorem in the context of involutions, and in so doing they must prove that $\phi^*(\sigma^{-1})=(\phi^*(\sigma))^{-1}$ for all permutations $\sigma$, so that the bijection from $S_n(1\ldots k)$ onto $S_n(k\ldots 1)$ will also commute with inverses.
The proof of the commutation result for $\phi^*$ is long and difficult, and Bousquet-M$\acute{\textrm{e}}$lou and Steingr$\acute{\textrm{\i}}$msson ask for a reformulation of $\phi^*$ that will make this result obvious.  In [6], Krattenthaler describes another bijection from $S_n(1\ldots k)$ into $S_n(k\ldots 1)$, in terms of growth diagrams, and notes that this bijection clearly commutes with inverses. He asks what connection there is between this bijection and $\phi^*$. In the present paper we answer both questions, by providing a reformulation of $\phi^*$ in terms of growth diagrams.

In proving their theorem, Backelin, West, and Xin find it necessary to work in the context of full rook placements on Ferrers boards, which includes permutations as a special case. For any Ferrers board $F$ and any permutation $\tau$, let $S_F(\tau)$ denote the set of all full rook placements on $F$ that avoid $\tau.$ (The relevant definitions will be reviewed in Section 2.)

Backelin, West, and Xin prove that
$|S_F(1\ldots k \rho)|=|S_F(k\ldots 1\rho)|$, for all $F$ and all permutations $\rho$ of $\{k+1,\ldots, k+\ell\}$, and in so doing they use an extension of $\phi^*$ to full rook placements. Bousquet-M$\acute{\textrm{e}}$lou and Steingr$\acute{\textrm{\i}}$msson also use this extension, so they prove that $\phi^*$ commutes with inverses in this broader context. Accordingly, our reformulation of $\phi^*$ will be given in this context, or rather in the even broader context of arbitrary rook placements (not necessarily full), with the term ``inverse" interpreted appropriately.

The outline of our paper is as follows. In Section 2 we review the needed background material on the Robinson-Schensted correspondence for partial permutations, Ferrers boards and rook placements, and growth diagrams. In Section 3 we give our reformulation of $\phi^*$ and the proof that it works, modulo a ``Main Lemma". In Section 4 we introduce the tool that will be used to prove this lemma: the  ``pivots" of a rook placement on a rectangular Ferrers board.  The pivots are related to the ``$L$-corners" of [2], and are a generalization of the ``$rcL$-corners" of  [8].  (We have chosen to use the term ``pivots", instead of ``corners", because of the prior use of the term ``corners" in connection with the diagram of a permutation.) Section 5 contains the proof of the Main Lemma, and the concluding Section 6 indicates how the proof leads naturally to a notion of generalized Knuth transformations.

\vspace{.25in}
\noindent\textbf{2. Review of the Needed Background}

\vspace{.25in}
\noindent\emph{2.1. The Robinson-Schensted correspondence for partial permutations}

\vspace{.15in}

A \emph{partial permutation} $\pi$ is a bijection between two sets $I$ and $J$ of positive integers. We represent $\pi$ in two-line notation as

\begin{displaymath}\pi = \left(\begin{array}{cccc}
 i_1 & i_2 & \cdots & i_m \\
 j_1 & j_2 & \cdots & j_m
\end{array}\right)\end{displaymath}
where $i_1< \ldots <i_m$ and the $j_t$'s are distinct.

The Robinson-Schensted \emph{insertion and recording tableaux} $P$ and $Q$ for $\pi$ are obtained as follows. Start by placing $j_1$ in the top row of $P$ and $i_1$ in the top row of $Q$. Assuming inductively that $i_1,\ldots, i_{t-1}$ and $j_1,\ldots, j_{t-1}$ have been placed, place $j_t$ in $P$ as follows. If $j_t$ is larger than all elements already in the top row of $P$, place $j_t$ at the right end of this row. If $j_s$ is the leftmost element already in the top row that is larger than $j_t$, replace $j_s$ by $j_t$ and ``bump" $j_s$ to the second row. Place $j_s$ at the right end of this row unless it is smaller than some element in this row, in which case let $j_s$ bump the leftmost element larger than it to the third row, and continue in this way. Place $i_t$ in the position in $Q$ corresponding to the position in $P$ that first became occupied when $j_t$ was placed in $P$.

For example, if
\begin{displaymath}\pi = \left(\begin{array}{ccccc}
 1 & 2 & 6 & 7 & 8 \\
 4 & 5 & 3 & 1 & 2
\end{array}\right)\end{displaymath}
then

\begin{center}$
\begin{array}{ccccc}
\raisebox{.6cm}{$P_\pi =$} &\young(12,35,4) &\textrm{\hspace{.5in}} &\raisebox{.6cm}{$Q_\pi =$} &\young(12,68,7)
\end{array}
$
\end{center}
 
 A basic property of $P$ and $Q$ is that there exists a partition $\lambda=\{\lambda_1\geq \lambda_2\geq \ldots \geq\lambda_t\}$ of $m$ such that the top rows of $P$ and $Q$ each contain $\lambda_1$ entries, the second rows each contain $\lambda_2$ entries, and so on. The partition $\lambda$ is called the \emph{shape} of $P$ and $Q$. In motivating our reformulation of $\phi^*$ it will be helpful to recall the theorem of Schensted [9] that states that $\lambda_1$ is the length of the longest increasing subsequence in $\pi$, and $t$ is the length of the longest decreasing subsequence.
 
 We will also use the theorem of Sch$\ddot{\textrm{u}}$tzenberger [10] that states that the insertion and recording tableaux for the inverse bijection $\pi^{-1}$ are $Q$ and $P$, respectively.

 \vspace{.15in}
 \noindent\emph{2.2. Ferrers boards, rook placements, and $\phi^*$}
 
 \vspace{.15in}
 Consider an $n\times n$ array of squares, and identify the pair $(i,j)$ with the square located in the $i$th column from the left and the $j$th row from the bottom. For any square $(i,j)$ in the array, let $R(i,j)$ denote the rectangle consisting of all squares $(k,\ell)$ such that $k\leq i$ and $\ell\leq j.$ A \emph{Ferrers board} (in French notation) is any subset $F$ of such an array with the property that for all $(i,j)\in F$ we have $R(i,j)\subseteq F.$ So for some $t$ and some $\lambda_1\geq\ldots \geq\lambda_t$, the Ferrers board consists of the first $\lambda_j$ squares from the $j$th row of the array, $1\leq j\leq t.$  The \emph{conjugate} of $F$ is the Ferrers board $F'=\{(j,i):(i,j)\in F\}$, so that $F'$ is obtained  by reflecting $F$ across the $SW$-$NE$ diagonal.
 
 A \emph{rook placement} on a Ferrers board $F$ is a subset of $F$ that contains at most one square from each row of $F$ and at most one square from each column of $F$.  We indicate the squares in the placement by putting markers (e.g., dots or $X$'s) in them. A rook placement is called \emph{full} if it includes exactly one square from each row and column of $F$.  (So if there exists a full rook placement on $F$, then $F$ has the same number of columns as rows.) From any rook placement $P$ on $F$ there results a partial permutation $\pi$ such that square $(i,j)$ is in $P$ if and only if $j$ is the value of the bijection $\pi$ at input $i$. $P$  is a full placement if and only if $F$ has $n$ rows and $n$ columns for some $n$ and $\pi$ is a permutation of $\{1,\ldots,n\}$. For any rook placement $P$ on $F$, the \emph{inverse} $P'$ of $P$ is the placement on the conjugate board $F'$ obtained by reflecting $F$ and all the markers for $P$ across the $SW$-$NE$ diagonal. The partial permutations resulting from $P$ and $P'$ are inverses of each other, if we regard them as bijections between sets.  If they are permutations, they are inverses in the usual sense.
 
 We say that a rook placement $P$ \emph{contains} a permutation $\tau\in S_r$ if and only if the resulting partial permutation $\pi$ contains a subsequence $\pi_{i_1}\ldots \pi_{i_r}$ order isomorphic to $\tau$ such that there is a rectangular subboard of $F$ that contains all the squares $(i_j,\pi_{i_j}).$ In this case we refer to the sequence of squares $(i_j,\pi_{i_j})$ as an \emph{occurrence} of $\tau$ in $P$. We say that $P$ \emph{avoids} $\tau$ if $P$ does not contain $\tau$.
 
 It is clear how to extend the definitions of $\phi$ and $\phi^*$ to rook placements, by using only the occurrences of $k\ldots 1$ in $P$, in the sense of the preceding paragraph.
 
 \vspace{.15in}
 
 \noindent\emph{2.3. Growth diagrams}
 
 \vspace{.15in}
 
 Our reformulation of $\phi^*$ will be accomplished by modifying Fomin's ([4, 5], see also [6]) construction of the growth diagram of a rook placement $P$ on a Ferrers board $F$.
 
 Fomin's construction assigns   partitions to the corners of all the squares in $F$, using the markers of $P$, in such a way that the partition assigned to any corner either equals the partition to its left or is obtained from it by adding 1 to one entry, and the partition assigned to any corner either equals the partition below it or is obtained from this partition by adding 1 to one entry. We start by assigning the empty partition $ \emptyset$ to each corner on the left and bottom edges of $F$. We then assign partitions to the other corners inductively. Assuming that the northwest, southwest, and southeast corners of a square $(i,j)$ have been assigned partitions $NW, SW,$ and $SE$, we assign to the northeast corner the partition $NE$ determined by the following rules.
 \begin{itemize}
  \item[1.] If $NW\neq SE$ then let $NE=NW \cup SE$, the partition whose $i$th entry is the maximum of the $i$th entries of $NW$ and $SE$. (Here we regard the absence of an entry as the presence of an entry $0$.)
 \item[2.] If $SW\neq NW=SE$ then $NW$ is obtained from $SW$ by adding 1 to the $i$th entry of $SW$, for some $i$. We obtain $NE$ from $NW$ by adding 1 to the $(i+1)$th entry.
 \item[3.] If $SW=NW=SE$ then we let $NE=SW$ unless the square $(i,j)$ contains a marker, in which case we obtain $NE$ from $SW$ by adding 1 to the first entry.
 \end{itemize}

\begin{lemma}\label{lemma1} ([6], Theorem 5.2.4 or [11], Theorem 7.13.5). For any square $(i,j)$, the partition assigned to the northeast corner of $(i,j)$ is the shape of the Robinson-Schensted tableaux for the partial permutation resulting from the restriction of $P$ to the rectangle $R(i,j).$
\end{lemma}

 \noindent \textbf{Example}. Consider the following (full) rook placement $P$ on the indicated Ferrers board $F$.
 
 
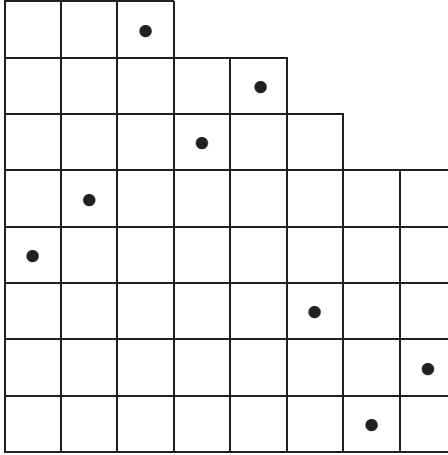
\vspace{.15in}
 
\begin{figure}[h]
\begin{center}
\ifx\JPicScale\undefined\def\JPicScale{.75}\fi
\unitlength \JPicScale mm
\begin{picture}(70,70)(0,0)
\linethickness{0.1mm}
\put(0,0){\line(1,0){80}}
\linethickness{0.1mm}
\put(0,80){\line(1,0){30}}
\linethickness{0.1mm}
\put(0,70){\line(1,0){50}}
\linethickness{0.1mm}
\put(0,50){\line(1,0){80}}
\linethickness{0.1mm}
\put(0,40){\line(1,0){80}}
\linethickness{0.1mm}
\put(0,30){\line(1,0){80}}
\linethickness{0.1mm}
\put(0,20){\line(1,0){80}}
\linethickness{0.1mm}
\put(0,10){\line(1,0){80}}
\linethickness{0.1mm}
\put(80,0){\line(0,1){50}}
\linethickness{0.1mm}
\put(50,0){\line(0,1){70}}
\linethickness{0.1mm}
\put(0,0){\line(0,1){80}}

\put(25,74.64){\makebox(0,0)[cc]{\large{$\bullet$}}}

\put(45.37,64.64){\makebox(0,0)[cc]{\large{$\bullet$}}}

\put(35,54.64){\makebox(0,0)[cc]{\large{$\bullet$}}}

\put(15,44.64){\makebox(0,0)[cc]{\large{$\bullet$}}}

\put(5,34.64){\makebox(0,0)[cc]{\large{$\bullet$}}}

\put(55,24.64){\makebox(0,0)[cc]{\large{$\bullet$}}}

\put(75.07,14.64){\makebox(0,0)[cc]{\large{$\bullet$}}}

\put(65.06,4.64){\makebox(0,0)[cc]{\large{$\bullet$}}}

\linethickness{0.1mm}
\put(70,0){\line(0,1){50}}
\linethickness{0.1mm}
\put(60,0){\line(0,1){60}}
\linethickness{0.1mm}
\put(0,60){\line(1,0){60}}
\linethickness{0.1mm}
\put(30,0){\line(0,1){80}}
\linethickness{0.1mm}
\put(40,0){\line(0,1){70}}
\linethickness{0.1mm}
\put(20,0){\line(0,1){80}}
\linethickness{0.1mm}
\put(10,0){\line(0,1){80}}
\end{picture}
\end{center}
\end{figure}
 
\vspace{.15in}

\noindent The growth diagram is (denoting each partition by juxtaposing its entries)

 
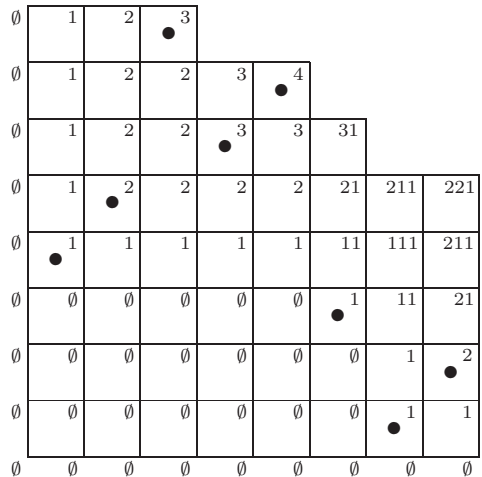
\vspace{.15in}
 
\begin{figure}[h]
\begin{center}
\ifx\JPicScale\undefined\def\JPicScale{.75}\fi
\unitlength \JPicScale mm
\begin{picture}(90,90)(0,0)
\linethickness{0.1mm}
\put(10,10){\line(1,0){80}}
\linethickness{0.1mm}
\put(10,90){\line(1,0){30}}
\linethickness{0.1mm}
\put(10,80){\line(1,0){50}}
\linethickness{0.1mm}
\put(10,60){\line(1,0){80}}
\linethickness{0.1mm}
\put(10,50){\line(1,0){80}}
\linethickness{0.1mm}
\put(10,40){\line(1,0){80}}
\linethickness{0.1mm}
\put(10,30){\line(1,0){80}}
\linethickness{0.1mm}
\put(10,20){\line(1,0){80}}
\linethickness{0.1mm}
\put(90,10){\line(0,1){50}}
\linethickness{0.1mm}
\put(60,10){\line(0,1){70}}
\linethickness{0.1mm}
\put(10,10){\line(0,1){80}}

\put(35,85){\makebox(0,0)[cc]{\large{$\bullet$}}}

\put(55,75){\makebox(0,0)[cc]{\large{$\bullet$}}}

\put(45,65){\makebox(0,0)[cc]{\large{$\bullet$}}}

\put(25,55){\makebox(0,0)[cc]{\large{$\bullet$}}}

\put(15,45){\makebox(0,0)[cc]{\large{$\bullet$}}}

\put(65,35){\makebox(0,0)[cc]{\large{$\bullet$}}}

\put(85,25){\makebox(0,0)[cc]{\large{$\bullet$}}}

\put(75,15){\makebox(0,0)[cc]{\large{$\bullet$}}}

\linethickness{0.1mm}
\put(80,10){\line(0,1){50}}
\linethickness{0.1mm}
\put(70,10){\line(0,1){60}}
\linethickness{0.1mm}
\put(10,70){\line(1,0){60}}
\linethickness{0.1mm}
\put(40,10){\line(0,1){80}}
\linethickness{0.1mm}
\put(50,10){\line(0,1){70}}
\linethickness{0.1mm}
\put(30,10){\line(0,1){80}}
\linethickness{0.1mm}
\put(20,10){\line(0,1){80}}

\put(18,48){\makebox(0,0)[cc]{\scriptsize{1}}}

\put(28,48){\makebox(0,0)[cc]{\scriptsize{1}}}

\put(48,48){\makebox(0,0)[cc]{\scriptsize{1}}}

\put(58,48){\makebox(0,0)[cc]{\scriptsize{1}}}

\put(38,48){\makebox(0,0)[cc]{\scriptsize{1}}}

\put(68,38){\makebox(0,0)[cc]{\scriptsize{1}}}

\put(78,28){\makebox(0,0)[cc]{\scriptsize{1}}}

\put(88,18){\makebox(0,0)[cc]{\scriptsize{1}}}

\put(78,18){\makebox(0,0)[cc]{\scriptsize{1}}}

\put(18,58){\makebox(0,0)[cc]{\scriptsize{1}}}

\put(18,68){\makebox(0,0)[cc]{\scriptsize{1}}}

\put(18,78){\makebox(0,0)[cc]{\scriptsize{1}}}

\put(18,88){\makebox(0,0)[cc]{\scriptsize{1}}}

\put(28,88){\makebox(0,0)[cc]{\scriptsize{2}}}

\put(28,78){\makebox(0,0)[cc]{\scriptsize{2}}}

\put(28,68){\makebox(0,0)[cc]{\scriptsize{2}}}

\put(28,58){\makebox(0,0)[cc]{\scriptsize{2}}}

\put(38,78){\makebox(0,0)[cc]{\scriptsize{2}}}

\put(38,68){\makebox(0,0)[cc]{\scriptsize{2}}}

\put(38,58){\makebox(0,0)[cc]{\scriptsize{2}}}

\put(48,58){\makebox(0,0)[cc]{\scriptsize{2}}}

\put(58,58){\makebox(0,0)[cc]{\scriptsize{2}}}

\put(88,28){\makebox(0,0)[cc]{\scriptsize{2}}}

\put(38,88){\makebox(0,0)[cc]{\scriptsize{3}}}

\put(48,78){\makebox(0,0)[cc]{\scriptsize{3}}}

\put(48,68){\makebox(0,0)[cc]{\scriptsize{3}}}

\put(58,68){\makebox(0,0)[cc]{\scriptsize{3}}}

\put(58,78){\makebox(0,0)[cc]{\scriptsize{4}}}

\put(67.38,48){\makebox(0,0)[cc]{\scriptsize{11}}}

\put(77.27,38){\makebox(0,0)[cc]{\scriptsize{11}}}

\put(87.38,38){\makebox(0,0)[cc]{\scriptsize{21}}}

\put(86.43,48){\makebox(0,0)[cc]{\scriptsize{211}}}

\put(76.3,58){\makebox(0,0)[cc]{\scriptsize{211}}}

\put(86.66,58){\makebox(0,0)[cc]{\scriptsize{221}}}

\put(76.54,48){\makebox(0,0)[cc]{\scriptsize{111}}}

\put(67.26,58){\makebox(0,0)[cc]{\scriptsize{21}}}

\put(66.9,68){\makebox(0,0)[cc]{\scriptsize{31}}}

\put(8,88){\makebox(0,0)[cc]{\scriptsize{$\emptyset$}}}

\put(8,78){\makebox(0,0)[cc]{\scriptsize{$\emptyset$}}}

\put(8,68){\makebox(0,0)[cc]{\scriptsize{$\emptyset$}}}

\put(8,58){\makebox(0,0)[cc]{\scriptsize{$\emptyset$}}}

\put(8,48){\makebox(0,0)[cc]{\scriptsize{$\emptyset$}}}

\put(8,38){\makebox(0,0)[cc]{\scriptsize{$\emptyset$}}}

\put(18,38){\makebox(0,0)[cc]{\scriptsize{$\emptyset$}}}

\put(18,28){\makebox(0,0)[cc]{\scriptsize{$\emptyset$}}}

\put(18,18){\makebox(0,0)[cc]{\scriptsize{$\emptyset$}}}

\put(28,38){\makebox(0,0)[cc]{\scriptsize{$\emptyset$}}}

\put(28,28){\makebox(0,0)[cc]{\scriptsize{$\emptyset$}}}

\put(28,18){\makebox(0,0)[cc]{\scriptsize{$\emptyset$}}}

\put(38,38){\makebox(0,0)[cc]{\scriptsize{$\emptyset$}}}

\put(38,28){\makebox(0,0)[cc]{\scriptsize{$\emptyset$}}}

\put(38,18){\makebox(0,0)[cc]{\scriptsize{$\emptyset$}}}

\put(48,38){\makebox(0,0)[cc]{\scriptsize{$\emptyset$}}}

\put(48,28){\makebox(0,0)[cc]{\scriptsize{$\emptyset$}}}

\put(48,18){\makebox(0,0)[cc]{\scriptsize{$\emptyset$}}}

\put(58,38){\makebox(0,0)[cc]{\scriptsize{$\emptyset$}}}

\put(58,28){\makebox(0,0)[cc]{\scriptsize{$\emptyset$}}}

\put(58,18){\makebox(0,0)[cc]{\scriptsize{$\emptyset$}}}

\put(8,28){\makebox(0,0)[cc]{\scriptsize{$\emptyset$}}}

\put(8,18){\makebox(0,0)[cc]{\scriptsize{$\emptyset$}}}

\put(68,28){\makebox(0,0)[cc]{\scriptsize{$\emptyset$}}}

\put(68,18){\makebox(0,0)[cc]{\scriptsize{$\emptyset$}}}

\put(18,8){\makebox(0,0)[cc]{\scriptsize{$\emptyset$}}}

\put(28,8){\makebox(0,0)[cc]{\scriptsize{$\emptyset$}}}

\put(38,8){\makebox(0,0)[cc]{\scriptsize{$\emptyset$}}}

\put(48,8){\makebox(0,0)[cc]{\scriptsize{$\emptyset$}}}

\put(58,8){\makebox(0,0)[cc]{\scriptsize{$\emptyset$}}}

\put(8,8){\makebox(0,0)[cc]{\scriptsize{$\emptyset$}}}

\put(68,8){\makebox(0,0)[cc]{\scriptsize{$\emptyset$}}}

\put(78,8){\makebox(0,0)[cc]{\scriptsize{$\emptyset$}}}

\put(88,8){\makebox(0,0)[cc]{\scriptsize{$\emptyset$}}}

\end{picture}
\end{center}
\end{figure}

 \noindent Note that the partial permutation obtained by restricting $P$ to $R(8,5)$ is
 
 \begin{displaymath}\pi = \left(\begin{array}{ccccc}
 1 & 2 & 6 & 7 & 8 \\
 4 & 5 & 3 & 1 & 2
\end{array}\right)\end{displaymath}
 
\noindent  and the partition assigned to the northeast corner of square $(8,5)$ is 221, which is, by the example given in subsection 2.1, the shape of the Robinson-Schensted tableaux for $\pi$.

In general, if we are given $F$ and the partitions in the growth diagram for $P$ that occur along the right/up border of $F$ (i.e., the border of $F$ minus the horizontal bottom edge and the vertical left edge), we can inductively reconstruct the rest of the growth diagram and the placement $P$ by using the following rules to assign a partition to the southwest corner of square $(i,j)$, given the partitions assigned to its other three corners.
 
 \begin{itemize}
 \item[A.] If $NW\neq SE$ then let $SW=NW\cap SE,$ the partition whose $i$th entry is the minimum of the $i$th entries of $NW$ and $SE.$
 \item[B.] If $NW=NE=SE$ then let $SW=NW.$
 \item[C.] If $NE\neq NW=SE$ and $NE$ differs from $NW$ in the $i$th entry for some $i\geq 2,$ then let $SW$ be obtained from $NW$ by subtracting 1 from the $(i-1)$th entry of $NW.$ If  $NE$ differs from $NW$ in the first entry, then let $SW=NW$ and, \emph{in this circumstance only}, place a marker in square $(i,j).$
 \end{itemize}
 
 Since the application of these rules recovers the growth diagram and the placement $P$, it follows that, in the growth diagram, square $(i,j)$ has a marker in it (i.e., is in $P$) if and only if $NE\neq NW=SE$ and $NE$ differs from $NW$ in the first entry.
 
 \vspace{.25in}
 %SECTION 3---3---3---3---3---3---3---3---3---3---3---3---3---3---3---3---3---3---3---3---3---3---3---3---3
 \noindent\textbf{3. The Reformulation of $\phi^*$}
 
 \vspace{.25in}
 
 We now modify the growth diagram algorithm ($GDA$) of subsection 2.3 to get a new algorithm $GDA_k$ for any $k\geq 2.$  $GDA_k$ retains rules (1) and (3) of $GDA$, but replaces rule (2) by the following variant.
 
 \vspace{.15in}
 
\noindent $2_k.$ Apply rule (2) with the proviso that if rule (2) produces a $NE$ with $k$ (nonzero) entries then delete the last entry and increase the first entry by 1.
 
 \vspace{.15in}
 
The motivation for rule $(2_k)$ comes from the theorem of Schensted mentioned in subsection 2.1. Keeping the number of entries in a partition $\lambda$ less than $k$ prevents decreasing subsequences of length $k$ in partial permutations whose Robinsion-Schensted tableaux have shape $\lambda$.

\begin{definitions}  
For any rook placement $P$ on a Ferrers board $F$, let seq$(P,F)$ (respectively, seq$_k(P,F)$) denote the sequence of partitions along the right/up border of $F$ that results from the application of $GDA$ (respectively, $GDA_k$) to $P$ and $F$.
 \end{definitions}
 
 \noindent\textbf{Main Theorem}. Fix $k\geq 2$ in the definition of $\phi^*$.  Then for any rook placement $P$ on a Ferrers board $F$, $$\textrm{seq}_k(P,F)=\textrm{seq}(\phi^*(P),F).$$
 
 \vspace{.15in}
 
 \noindent\textbf{Corollary}. For any rook placement $P$ on a Ferrers board $F$, $$\phi^*(P')=(\phi^*(P))'.$$
 \emph{Proof.} This is essentially clear from the fact that the algorithms $GDA$ and $GDA_k$ commute with the operation of taking the inverse of a placement.
 
 In a bit more detail, seq$_k(P',F')$ is the reverse of seq$_k(P,F)$, so, by the Main Theorem, seq$(\phi^*(P'),F')$ is the reverse of seq$(\phi^*(P),F)$, and this reverse is seq$((\phi^*(P))',F')$.  By rules (A), (B), and (C) for the inverse algorithm for $GDA$,  we conclude that $\phi^*(P')=(\phi^*(P))'.$   $\Box$
 
 \vspace{.15in}
 
 We will now give an example, to illustrate the Main Theorem and indicate the structure of its proof.
 
 \vspace{.15in}
 
 \noindent\textbf{Example.}  Let $P$ be the placement from the example in subsection 2.3, and fix $k=3$ in the definition of $\phi^*.$ Performing $GDA_3$ on $P$ yields

\begin{figure}[h]
\begin{center}
\ifx\JPicScale\undefined\def\JPicScale{.75}\fi
\unitlength \JPicScale mm
\begin{picture}(90,90)(0,0)
\linethickness{0.1mm}
\put(10,10){\line(1,0){80}}
\linethickness{0.1mm}
\put(10,90){\line(1,0){30}}
\linethickness{0.1mm}
\put(10,80){\line(1,0){50}}
\linethickness{0.1mm}
\put(10,60){\line(1,0){80}}
\linethickness{0.1mm}
\put(10,50){\line(1,0){80}}
\linethickness{0.1mm}
\put(10,40){\line(1,0){80}}
\linethickness{0.1mm}
\put(10,30){\line(1,0){80}}
\linethickness{0.1mm}
\put(10,20){\line(1,0){80}}
\linethickness{0.1mm}
\put(90,10){\line(0,1){50}}
\linethickness{0.1mm}
\put(60,10){\line(0,1){70}}
\linethickness{0.1mm}
\put(10,10){\line(0,1){80}}
\put(35,84.64){\makebox(0,0)[cc]{\large{$\bullet$}}}

\put(55.37,74.63){\makebox(0,0)[cc]{\large{$\bullet$}}}

\put(45,65){\makebox(0,0)[cc]{\large{$\bullet$}}}

\put(25,55){\makebox(0,0)[cc]{\large{$\bullet$}}}

\put(15,45){\makebox(0,0)[cc]{\large{$\bullet$}}}

\put(65,34.64){\makebox(0,0)[cc]{\large{$\bullet$}}}

\put(75.06,14.22){\makebox(0,0)[cc]{\large{$\bullet$}}}

\put(85.07,24.94){\makebox(0,0)[cc]{\large{$\bullet$}}}

\linethickness{0.1mm}
\put(80,10){\line(0,1){50}}
\linethickness{0.1mm}
\put(70,10){\line(0,1){60}}
\linethickness{0.1mm}
\put(10,70){\line(1,0){60}}
\linethickness{0.1mm}
\put(40,10){\line(0,1){80}}
\linethickness{0.1mm}
\put(50,10){\line(0,1){70}}
\linethickness{0.1mm}
\put(30,10){\line(0,1){80}}
\linethickness{0.1mm}
\put(20,10){\line(0,1){80}}

\put(18,48){\makebox(0,0)[cc]{\scriptsize{1}}}

\put(28,48){\makebox(0,0)[cc]{\scriptsize{1}}}

\put(48,48){\makebox(0,0)[cc]{\scriptsize{1}}}

\put(58,48){\makebox(0,0)[cc]{\scriptsize{1}}}

\put(38,48){\makebox(0,0)[cc]{\scriptsize{1}}}

\put(68.21,38){\makebox(0,0)[cc]{\scriptsize{1}}}

\put(78,28){\makebox(0,0)[cc]{\scriptsize{1}}}

\put(88,18){\makebox(0,0)[cc]{\scriptsize{1}}}

\put(78,18){\makebox(0,0)[cc]{\scriptsize{1}}}

\put(18,58){\makebox(0,0)[cc]{\scriptsize{1}}}

\put(18,68){\makebox(0,0)[cc]{\scriptsize{1}}}

\put(18,78){\makebox(0,0)[cc]{\scriptsize{1}}}

\put(18,88){\makebox(0,0)[cc]{\scriptsize{1}}}

\put(28,88){\makebox(0,0)[cc]{\scriptsize{2}}}

\put(28,78){\makebox(0,0)[cc]{\scriptsize{2}}}

\put(28,68){\makebox(0,0)[cc]{\scriptsize{2}}}

\put(28,58){\makebox(0,0)[cc]{\scriptsize{2}}}

\put(38,78){\makebox(0,0)[cc]{\scriptsize{2}}}

\put(38,68){\makebox(0,0)[cc]{\scriptsize{2}}}

\put(38,58){\makebox(0,0)[cc]{\scriptsize{2}}}

\put(48,58){\makebox(0,0)[cc]{\scriptsize{2}}}

\put(58,58){\makebox(0,0)[cc]{\scriptsize{2}}}

\put(88,28){\makebox(0,0)[cc]{\scriptsize{2}}}

\put(38,88){\makebox(0,0)[cc]{\scriptsize{3}}}

\put(48,78){\makebox(0,0)[cc]{\scriptsize{3}}}

\put(48,68){\makebox(0,0)[cc]{\scriptsize{3}}}

\put(58,68){\makebox(0,0)[cc]{\scriptsize{3}}}

\put(58,78){\makebox(0,0)[cc]{\scriptsize{4}}}

\put(67.38,48){\makebox(0,0)[cc]{\scriptsize{11}}}

\put(77.27,38){\makebox(0,0)[cc]{\scriptsize{11}}}

\put(87.38,38){\makebox(0,0)[cc]{\scriptsize{21}}}

\put(87.38,48){\makebox(0,0)[cc]{\scriptsize{22}}}

\put(77.38,58){\makebox(0,0)[cc]{\scriptsize{22}}}

\put(87.38,58){\makebox(0,0)[cc]{\scriptsize{32}}}

\put(77.38,48){\makebox(0,0)[cc]{\scriptsize{21}}}

\put(67.38,58){\makebox(0,0)[cc]{\scriptsize{21}}}

\put(67.38,68){\makebox(0,0)[cc]{\scriptsize{31}}}

\put(8,88.45){\makebox(0,0)[cc]{\scriptsize{$\emptyset$}}}

\put(8,77.96){\makebox(0,0)[cc]{\scriptsize{$\emptyset$}}}

\put(8,68.07){\makebox(0,0)[cc]{\scriptsize{$\emptyset$}}}

\put(8,58.32){\makebox(0,0)[cc]{\scriptsize{$\emptyset$}}}

\put(8,47.83){\makebox(0,0)[cc]{\scriptsize{$\emptyset$}}}

\put(8,38){\makebox(0,0)[cc]{\scriptsize{$\emptyset$}}}

\put(18,38){\makebox(0,0)[cc]{\scriptsize{$\emptyset$}}}

\put(18,28){\makebox(0,0)[cc]{\scriptsize{$\emptyset$}}}

\put(18,18){\makebox(0,0)[cc]{\scriptsize{$\emptyset$}}}

\put(28,38){\makebox(0,0)[cc]{\scriptsize{$\emptyset$}}}

\put(28,28){\makebox(0,0)[cc]{\scriptsize{$\emptyset$}}}

\put(28,18){\makebox(0,0)[cc]{\scriptsize{$\emptyset$}}}

\put(38,38){\makebox(0,0)[cc]{\scriptsize{$\emptyset$}}}

\put(38,28){\makebox(0,0)[cc]{\scriptsize{$\emptyset$}}}

\put(38,18){\makebox(0,0)[cc]{\scriptsize{$\emptyset$}}}

\put(48,38){\makebox(0,0)[cc]{\scriptsize{$\emptyset$}}}

\put(48,28){\makebox(0,0)[cc]{\scriptsize{$\emptyset$}}}

\put(48,18){\makebox(0,0)[cc]{\scriptsize{$\emptyset$}}}

\put(58,38){\makebox(0,0)[cc]{\scriptsize{$\emptyset$}}}

\put(58,28){\makebox(0,0)[cc]{\scriptsize{$\emptyset$}}}

\put(58,18){\makebox(0,0)[cc]{\scriptsize{$\emptyset$}}}

\put(8,28){\makebox(0,0)[cc]{\scriptsize{$\emptyset$}}}

\put(8,18){\makebox(0,0)[cc]{\scriptsize{$\emptyset$}}}

\put(68,28){\makebox(0,0)[cc]{\scriptsize{$\emptyset$}}}

\put(68,18){\makebox(0,0)[cc]{\scriptsize{$\emptyset$}}}

\put(18,8){\makebox(0,0)[cc]{\scriptsize{$\emptyset$}}}

\put(28,8){\makebox(0,0)[cc]{\scriptsize{$\emptyset$}}}

\put(38,8){\makebox(0,0)[cc]{\scriptsize{$\emptyset$}}}

\put(48,8){\makebox(0,0)[cc]{\scriptsize{$\emptyset$}}}

\put(58,8){\makebox(0,0)[cc]{\scriptsize{$\emptyset$}}}

\put(8,8){\makebox(0,0)[cc]{\scriptsize{$\emptyset$}}}

\put(68,8){\makebox(0,0)[cc]{\scriptsize{$\emptyset$}}}

\put(78,8){\makebox(0,0)[cc]{\scriptsize{$\emptyset$}}}

\put(88,8){\makebox(0,0)[cc]{\scriptsize{$\emptyset$}}}

\end{picture}
\end{center}
\end{figure}
 
 \noindent and performing $GDA$ on $\phi^*(P)$ yields
 
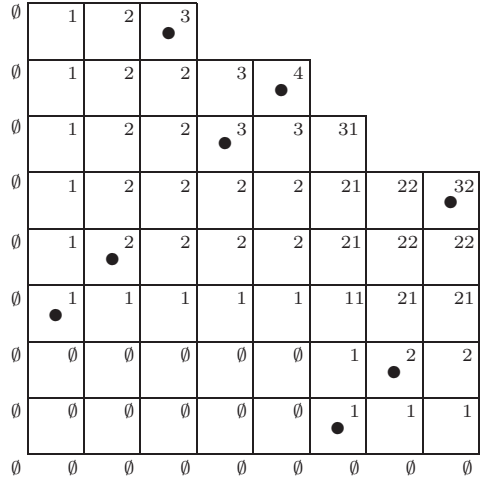
\begin{figure}[h!]
\begin{center}
\ifx\JPicScale\undefined\def\JPicScale{.75}\fi
\unitlength \JPicScale mm
\begin{picture}(90,100)(0,0)
\linethickness{0.1mm}
\put(10,10){\line(1,0){80}}
\linethickness{0.1mm}
\put(10,90){\line(1,0){30}}
\linethickness{0.1mm}
\put(10,80){\line(1,0){50}}
\linethickness{0.1mm}
\put(10,60){\line(1,0){80}}
\linethickness{0.1mm}
\put(10,50){\line(1,0){80}}
\linethickness{0.1mm}
\put(10,40){\line(1,0){80}}
\linethickness{0.1mm}
\put(10,30){\line(1,0){80}}
\linethickness{0.1mm}
\put(10,20){\line(1,0){80}}
\linethickness{0.1mm}
\put(90,10){\line(0,1){50}}
\linethickness{0.1mm}
\put(60,10){\line(0,1){70}}
\linethickness{0.1mm}
\put(10,10){\line(0,1){80}}

\put(15,34.5){\makebox(0,0)[cc]{\large{$\bullet$}}}
\put(25,44.5){\makebox(0,0)[cc]{\large{$\bullet$}}}
\put(35,84.5){\makebox(0,0)[cc]{\large{$\bullet$}}}
\put(45,65){\makebox(0,0)[cc]{\large{$\bullet$}}}
\put(55,74.5){\makebox(0,0)[cc]{\large{$\bullet$}}}
\put(65,14.5){\makebox(0,0)[cc]{\large{$\bullet$}}}
\put(75,24.5){\makebox(0,0)[cc]{\large{$\bullet$}}}
\put(85,54.5){\makebox(0,0)[cc]{\large{$\bullet$}}}

\linethickness{0.1mm}
\put(80,10){\line(0,1){50}}
\linethickness{0.1mm}
\put(70,10){\line(0,1){60}}
\linethickness{0.1mm}
\put(10,70){\line(1,0){60}}
\linethickness{0.1mm}
\put(40,10){\line(0,1){80}}
\linethickness{0.1mm}
\put(50,10){\line(0,1){70}}
\linethickness{0.1mm}
\put(30,10){\line(0,1){80}}
\linethickness{0.1mm}
\put(20,10){\line(0,1){80}}

\put(8,7.67){\makebox(0,0)[cc]{\scriptsize{$\emptyset$}}}
\put(18,7.67){\makebox(0,0)[cc]{\scriptsize{$\emptyset$}}}
\put(28,7.67){\makebox(0,0)[cc]{\scriptsize{$\emptyset$}}}
\put(38,7.67){\makebox(0,0)[cc]{\scriptsize{$\emptyset$}}}
\put(48,7.67){\makebox(0,0)[cc]{\scriptsize{$\emptyset$}}}
\put(58,7.67){\makebox(0,0)[cc]{\scriptsize{$\emptyset$}}}
\put(68,7.67){\makebox(0,0)[cc]{\scriptsize{$\emptyset$}}}
\put(78,7.67){\makebox(0,0)[cc]{\scriptsize{$\emptyset$}}}
\put(88,7.67){\makebox(0,0)[cc]{\scriptsize{$\emptyset$}}}

\put(8,17.67){\makebox(0,0)[cc]{\scriptsize{$\emptyset$}}}
\put(18,17.67){\makebox(0,0)[cc]{\scriptsize{$\emptyset$}}}
\put(28,17.67){\makebox(0,0)[cc]{\scriptsize{$\emptyset$}}}
\put(38,17.67){\makebox(0,0)[cc]{\scriptsize{$\emptyset$}}}
\put(48,17.67){\makebox(0,0)[cc]{\scriptsize{$\emptyset$}}}
\put(58,17.67){\makebox(0,0)[cc]{\scriptsize{$\emptyset$}}}
\put(68,17.67){\makebox(0,0)[cc]{\scriptsize{1}}}
\put(78,17.67){\makebox(0,0)[cc]{\scriptsize{1}}}
\put(88,17.67){\makebox(0,0)[cc]{\scriptsize{1}}}

\put(8,27.67){\makebox(0,0)[cc]{\scriptsize{$\emptyset$}}}
\put(18,27.67){\makebox(0,0)[cc]{\scriptsize{$\emptyset$}}}
\put(28,27.67){\makebox(0,0)[cc]{\scriptsize{$\emptyset$}}}
\put(38,27.67){\makebox(0,0)[cc]{\scriptsize{$\emptyset$}}}
\put(48,27.67){\makebox(0,0)[cc]{\scriptsize{$\emptyset$}}}
\put(58,27.67){\makebox(0,0)[cc]{\scriptsize{$\emptyset$}}}
\put(68,27.67){\makebox(0,0)[cc]{\scriptsize{1}}}
\put(78,27.67){\makebox(0,0)[cc]{\scriptsize{2}}}
\put(88,27.67){\makebox(0,0)[cc]{\scriptsize{2}}}

\put(8,37.67){\makebox(0,0)[cc]{\scriptsize{$\emptyset$}}}
\put(18,37.67){\makebox(0,0)[cc]{\scriptsize{1}}}
\put(28,37.67){\makebox(0,0)[cc]{\scriptsize{1}}}
\put(38,37.67){\makebox(0,0)[cc]{\scriptsize{1}}}
\put(48,37.67){\makebox(0,0)[cc]{\scriptsize{1}}}
\put(58,37.67){\makebox(0,0)[cc]{\scriptsize{1}}}
\put(68,37.67){\makebox(0,0)[cc]{\scriptsize{11}}}
\put(77.4,37.67){\makebox(0,0)[cc]{\scriptsize{21}}}
\put(87.4,37.67){\makebox(0,0)[cc]{\scriptsize{21}}}

\put(8,47.83){\makebox(0,0)[cc]{\scriptsize{$\emptyset$}}}
\put(18,47.67){\makebox(0,0)[cc]{\scriptsize{1}}}
\put(28,47.67){\makebox(0,0)[cc]{\scriptsize{2}}}
\put(38,47.67){\makebox(0,0)[cc]{\scriptsize{2}}}
\put(48,47.67){\makebox(0,0)[cc]{\scriptsize{2}}}
\put(58,47.67){\makebox(0,0)[cc]{\scriptsize{2}}}
\put(67.4,47.67){\makebox(0,0)[cc]{\scriptsize{21}}}
\put(77.4,47.67){\makebox(0,0)[cc]{\scriptsize{22}}}
\put(87.4,47.67){\makebox(0,0)[cc]{\scriptsize{22}}}

\put(8,58.32){\makebox(0,0)[cc]{\scriptsize{$\emptyset$}}}
\put(18,57.67){\makebox(0,0)[cc]{\scriptsize{1}}}
\put(28,57.67){\makebox(0,0)[cc]{\scriptsize{2}}}
\put(38,57.67){\makebox(0,0)[cc]{\scriptsize{2}}}
\put(48,57.67){\makebox(0,0)[cc]{\scriptsize{2}}}
\put(58,57.67){\makebox(0,0)[cc]{\scriptsize{2}}}
\put(67.4,57.67){\makebox(0,0)[cc]{\scriptsize{21}}}
\put(77.4,57.67){\makebox(0,0)[cc]{\scriptsize{22}}}
\put(87.4,57.67){\makebox(0,0)[cc]{\scriptsize{32}}}

\put(8,68.07){\makebox(0,0)[cc]{\scriptsize{$\emptyset$}}}
\put(18,67.67){\makebox(0,0)[cc]{\scriptsize{1}}}
\put(28,67.67){\makebox(0,0)[cc]{\scriptsize{2}}}
\put(38,67.67){\makebox(0,0)[cc]{\scriptsize{2}}}
\put(48,67.67){\makebox(0,0)[cc]{\scriptsize{3}}}
\put(58,67.67){\makebox(0,0)[cc]{\scriptsize{3}}}
\put(66.9,67.67){\makebox(0,0)[cc]{\scriptsize{31}}}

\put(8,77.96){\makebox(0,0)[cc]{\scriptsize{$\emptyset$}}}
\put(18,77.67){\makebox(0,0)[cc]{\scriptsize{1}}}
\put(28,77.67){\makebox(0,0)[cc]{\scriptsize{2}}}
\put(38,77.67){\makebox(0,0)[cc]{\scriptsize{2}}}
\put(48,77.67){\makebox(0,0)[cc]{\scriptsize{3}}}
\put(58,77.67){\makebox(0,0)[cc]{\scriptsize{4}}}

\put(8,88.45){\makebox(0,0)[cc]{\scriptsize{$\emptyset$}}}
\put(18,87.67){\makebox(0,0)[cc]{\scriptsize{1}}}
\put(28,87.67){\makebox(0,0)[cc]{\scriptsize{2}}}
\put(38,87.67){\makebox(0,0)[cc]{\scriptsize{3}}}

\end{picture}
\end{center}
\end{figure}
 
\newpage
 
\noindent The partitions along the right/up border of $F$ are the same in both cases, so seq$_3(P,F)$=seq$(\phi^*(P),F)$, although the partitions in the interiors of the diagrams are not always the same. 

In this example, performing $GDA_3$ on $\phi(P)$ yields
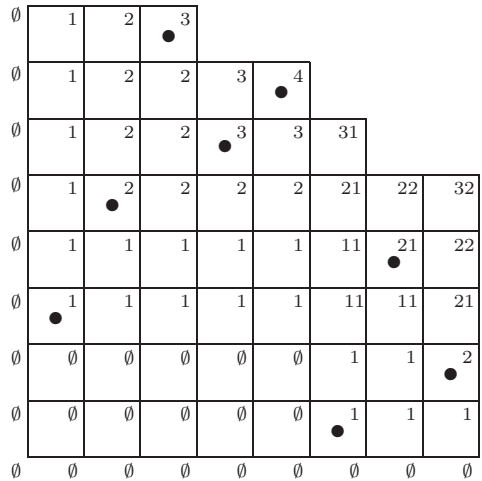
\begin{figure}[h!]
\begin{center}
\ifx\JPicScale\undefined\def\JPicScale{.75}\fi
\unitlength \JPicScale mm
\begin{picture}(90,100)(0,0)
\linethickness{0.1mm}
\put(10,10){\line(1,0){80}}
\linethickness{0.1mm}
\put(10,90){\line(1,0){30}}
\linethickness{0.1mm}
\put(10,80){\line(1,0){50}}
\linethickness{0.1mm}
\put(10,60){\line(1,0){80}}
\linethickness{0.1mm}
\put(10,50){\line(1,0){80}}
\linethickness{0.1mm}
\put(10,40){\line(1,0){80}}
\linethickness{0.1mm}
\put(10,30){\line(1,0){80}}
\linethickness{0.1mm}
\put(10,20){\line(1,0){80}}
\linethickness{0.1mm}
\put(90,10){\line(0,1){50}}
\linethickness{0.1mm}
\put(60,10){\line(0,1){70}}
\linethickness{0.1mm}
\put(10,10){\line(0,1){80}}
\put(35,84.5){\makebox(0,0)[cc]{\large{$\bullet$}}}

\put(55,74.5){\makebox(0,0)[cc]{\large{$\bullet$}}}

\put(45,65){\makebox(0,0)[cc]{\large{$\bullet$}}}

\put(25,54.5){\makebox(0,0)[cc]{\large{$\bullet$}}}

\put(15,34.5){\makebox(0,0)[cc]{\large{$\bullet$}}}

\put(75,44.5){\makebox(0,0)[cc]{\large{$\bullet$}}}

\put(65,14.5){\makebox(0,0)[cc]{\large{$\bullet$}}}

\put(85,24.5){\makebox(0,0)[cc]{\large{$\bullet$}}}

\linethickness{0.1mm}
\put(80,10){\line(0,1){50}}
\linethickness{0.1mm}
\put(70,10){\line(0,1){60}}
\linethickness{0.1mm}
\put(10,70){\line(1,0){60}}
\linethickness{0.1mm}
\put(40,10){\line(0,1){80}}
\linethickness{0.1mm}
\put(50,10){\line(0,1){70}}
\linethickness{0.1mm}
\put(30,10){\line(0,1){80}}
\linethickness{0.1mm}
\put(20,10){\line(0,1){80}}

\put(8,7.67){\makebox(0,0)[cc]{\scriptsize{$\emptyset$}}}
\put(18,7.67){\makebox(0,0)[cc]{\scriptsize{$\emptyset$}}}
\put(28,7.67){\makebox(0,0)[cc]{\scriptsize{$\emptyset$}}}
\put(38,7.67){\makebox(0,0)[cc]{\scriptsize{$\emptyset$}}}
\put(48,7.67){\makebox(0,0)[cc]{\scriptsize{$\emptyset$}}}
\put(58,7.67){\makebox(0,0)[cc]{\scriptsize{$\emptyset$}}}
\put(68,7.67){\makebox(0,0)[cc]{\scriptsize{$\emptyset$}}}
\put(78,7.67){\makebox(0,0)[cc]{\scriptsize{$\emptyset$}}}
\put(88,7.67){\makebox(0,0)[cc]{\scriptsize{$\emptyset$}}}

\put(8,17.67){\makebox(0,0)[cc]{\scriptsize{$\emptyset$}}}
\put(18,17.67){\makebox(0,0)[cc]{\scriptsize{$\emptyset$}}}
\put(28,17.67){\makebox(0,0)[cc]{\scriptsize{$\emptyset$}}}
\put(38,17.67){\makebox(0,0)[cc]{\scriptsize{$\emptyset$}}}
\put(48,17.67){\makebox(0,0)[cc]{\scriptsize{$\emptyset$}}}
\put(58,17.67){\makebox(0,0)[cc]{\scriptsize{$\emptyset$}}}
\put(68,17.67){\makebox(0,0)[cc]{\scriptsize{1}}}
\put(78,17.67){\makebox(0,0)[cc]{\scriptsize{1}}}
\put(88,17.67){\makebox(0,0)[cc]{\scriptsize{1}}}

\put(8,27.67){\makebox(0,0)[cc]{\scriptsize{$\emptyset$}}}
\put(18,27.67){\makebox(0,0)[cc]{\scriptsize{$\emptyset$}}}
\put(28,27.67){\makebox(0,0)[cc]{\scriptsize{$\emptyset$}}}
\put(38,27.67){\makebox(0,0)[cc]{\scriptsize{$\emptyset$}}}
\put(48,27.67){\makebox(0,0)[cc]{\scriptsize{$\emptyset$}}}
\put(58,27.67){\makebox(0,0)[cc]{\scriptsize{$\emptyset$}}}
\put(68,27.67){\makebox(0,0)[cc]{\scriptsize{1}}}
\put(78,27.67){\makebox(0,0)[cc]{\scriptsize{1}}}
\put(88,27.67){\makebox(0,0)[cc]{\scriptsize{2}}}

\put(8,37.67){\makebox(0,0)[cc]{\scriptsize{$\emptyset$}}}
\put(18,37.67){\makebox(0,0)[cc]{\scriptsize{1}}}
\put(28,37.67){\makebox(0,0)[cc]{\scriptsize{1}}}
\put(38,37.67){\makebox(0,0)[cc]{\scriptsize{1}}}
\put(48,37.67){\makebox(0,0)[cc]{\scriptsize{1}}}
\put(58,37.67){\makebox(0,0)[cc]{\scriptsize{1}}}
\put(68,37.67){\makebox(0,0)[cc]{\scriptsize{11}}}
\put(77.4,37.67){\makebox(0,0)[cc]{\scriptsize{11}}}
\put(87.4,37.67){\makebox(0,0)[cc]{\scriptsize{21}}}

\put(8,47.83){\makebox(0,0)[cc]{\scriptsize{$\emptyset$}}}
\put(18,47.67){\makebox(0,0)[cc]{\scriptsize{1}}}
\put(28,47.67){\makebox(0,0)[cc]{\scriptsize{1}}}
\put(38,47.67){\makebox(0,0)[cc]{\scriptsize{1}}}
\put(48,47.67){\makebox(0,0)[cc]{\scriptsize{1}}}
\put(58,47.67){\makebox(0,0)[cc]{\scriptsize{1}}}
\put(67.4,47.67){\makebox(0,0)[cc]{\scriptsize{11}}}
\put(77.4,47.67){\makebox(0,0)[cc]{\scriptsize{21}}}
\put(87.4,47.67){\makebox(0,0)[cc]{\scriptsize{22}}}

\put(8,58.32){\makebox(0,0)[cc]{\scriptsize{$\emptyset$}}}
\put(18,57.67){\makebox(0,0)[cc]{\scriptsize{1}}}
\put(28,57.67){\makebox(0,0)[cc]{\scriptsize{2}}}
\put(38,57.67){\makebox(0,0)[cc]{\scriptsize{2}}}
\put(48,57.67){\makebox(0,0)[cc]{\scriptsize{2}}}
\put(58,57.67){\makebox(0,0)[cc]{\scriptsize{2}}}
\put(67.4,57.67){\makebox(0,0)[cc]{\scriptsize{21}}}
\put(77.4,57.67){\makebox(0,0)[cc]{\scriptsize{22}}}
\put(87.4,57.67){\makebox(0,0)[cc]{\scriptsize{32}}}

\put(8,68.07){\makebox(0,0)[cc]{\scriptsize{$\emptyset$}}}
\put(18,67.67){\makebox(0,0)[cc]{\scriptsize{1}}}
\put(28,67.67){\makebox(0,0)[cc]{\scriptsize{2}}}
\put(38,67.67){\makebox(0,0)[cc]{\scriptsize{2}}}
\put(48,67.67){\makebox(0,0)[cc]{\scriptsize{3}}}
\put(58,67.67){\makebox(0,0)[cc]{\scriptsize{3}}}
\put(66.9,67.67){\makebox(0,0)[cc]{\scriptsize{31}}}

\put(8,77.96){\makebox(0,0)[cc]{\scriptsize{$\emptyset$}}}
\put(18,77.67){\makebox(0,0)[cc]{\scriptsize{1}}}
\put(28,77.67){\makebox(0,0)[cc]{\scriptsize{2}}}
\put(38,77.67){\makebox(0,0)[cc]{\scriptsize{2}}}
\put(48,77.67){\makebox(0,0)[cc]{\scriptsize{3}}}
\put(58,77.67){\makebox(0,0)[cc]{\scriptsize{4}}}

\put(8,88.45){\makebox(0,0)[cc]{\scriptsize{$\emptyset$}}}
\put(18,87.67){\makebox(0,0)[cc]{\scriptsize{1}}}
\put(28,87.67){\makebox(0,0)[cc]{\scriptsize{2}}}
\put(38,87.67){\makebox(0,0)[cc]{\scriptsize{3}}}

\end{picture}
\end{center}
\end{figure}

\noindent Notice that, while the results of performing $GDA_3$ on $P$ and on $\phi(P)$ are not the same diagram, they do agree on the boundary of $R(7,4)$, the smallest rectangular subboard of $F$ that contains the 321-pattern on which $\phi$ acted and extends to the left and bottom edges of $F$. Because of the definition of $GDA_3$, this is enough to make the two diagrams agree everywhere outside the rectangle. The idea of the proof of the Main Theorem will be to show that performing $GDA_k$ on $P$ and on $\phi(P)$ yields the same partitions on the boundary of the smallest rectangular subboard of $F$ that contains the $k\ldots1$-pattern on which $\phi$ acted and extends to the left and bottom edges of $F$.

\vspace{.15in}

\noindent\emph{Proof of the Main Theorem.}   We proceed by induction on the number of applications of $\phi$ required to compute $\phi^*(P).$  If no applications are required, the result is obvious, since $\phi^*(P)=P$ and performing $GDA_k$ is the same as performing $GDA.$

Now suppose that $m$ applications of $\phi$ are required to compute $\phi^*(P).$ Since computing $\phi^*(\phi(P))$ requires only $m-1$ applications of $\phi$, we assume inductively that$$\textrm{seq}_k(\phi(P),F)=\textrm{seq}(\phi^*(\phi(P)),F),$$
i.e., that
$$\textrm{seq}_k(\phi(P),F)=\textrm{seq}(\phi^*(P),F).$$
We want to show that $\textrm{seq}_k(\phi(P),F)=\textrm{seq}_k(P,F).$

Let $R=R(a,b)$ be the smallest rectangular subboard of $F$ that contains the $k\ldots1$-pattern on which $\phi$ acted to produce $\phi(P)$ and extends to the left and bottom edges of $F$.  Let $P_R$ and $\phi(P)_R$ be the restrictions of $P$ and $\phi(P)$ to $R(a,b).$  By the definition of $GDA_k$, all we need to show is that 
$$\textrm{seq}_k(\phi(P)_R,R)=\textrm{seq}_k(P_R,R).$$  To show this, we will use the following two lemmas, which will be proved at the end of this section.

\begin{lemma}\label{lemma2}
The placement $P_R$ on $R$ contains no $k\ldots1$-pattern that begins in a row below the top row of  $R$ or ends in a column to the left of the rightmost column of $R$.
\end{lemma}

\begin{lemma}\label{lemma3}
The placement $\phi(P)_R$ on $R$ contains no $k\ldots1$-pattern.
\end{lemma}

 By Lemma \ref{lemma3}, $$\textrm{seq}_k(\phi(P)_R,R)=\textrm{seq}(\phi(P)_R,R).$$ By Lemma \ref{lemma2},  $\textrm{seq}_k(P_R,R)=\textrm{seq}(P_R,R)$ except possibly at the northeast corner of square $(a,b).$
 
 \vspace{.15in}
 \noindent\textbf{Notation.}   Let $c_{ne}$, $c_{nw}$, and $c_{se}$ denote the northeast, northwest, and southeast corners of square $(a,b).$
 
 \vspace{.15in}
 To conclude the proof, it will suffice to prove the next lemma.
 
 \vspace{.15in}
 
 \noindent\textbf{Main Lemma.}  We have $\textrm{seq}(P_R,R)=\textrm{seq}(\phi(P)_R,R)$ except possibly at $c_{ne}$, the northeast corner of $R$.
 
 \vspace{.15in}
 
 For once this lemma is established, we will have $$\textrm{seq}_k(\phi(P)_R,R)=\textrm{seq}_k(P_R,R)$$ except possibly at $c_{ne}.$ To see that the two also agree at $c_{ne}$, let, by the Main Lemma, $\lambda$ be the common value of $\textrm{seq}(P_R,R)$ and $\textrm{seq}(\phi(P)_R,R)$ at $c_{nw}.$ Since$$
 \textrm{seq}_k(\phi(P)_R,R)=\textrm{seq}(\phi(P)_R,R),$$
 the value of $\textrm{seq}_k(\phi(P)_R,R)$ at $c_{ne}$ is obtained from $\lambda$ by adding 1 to the first entry. (This follows from the last statement in subsection 2.3 and the fact that $\phi(P)_R$ has a marker in square $(a,b).$) To see that $\textrm{seq}_k(P_R,R)$ has the same value at $c_{ne}$, let $\mu$, $\nu$ be the values of $\textrm{seq}(P_R,R)$ at $c_{se}$, $c_{ne}.$ Since, in the placement $P_R$ on $R$, $R(a-1,b)$ and $R(a,b-1)$  contain no $k\ldots 1$-patterns (by Lemma \ref{lemma2}) but $R(a,b)$ contains such a pattern, it must be that $\nu$ has $k$ entries and each of $\lambda$, $\mu$ has $k-1.$ Therefore, by the definition of $GDA_k$, the value of $\textrm{seq}_k(P_R,R)$ at $c_{ne}$ is obtained by adding 1 to the first entry of $\lambda.$
 
 \vspace{.15in}
 
 The proof of the Main Theorem is now complete, except for the proofs of Lemmas \ref{lemma2} and \ref{lemma3} and the Main Lemma.  The proof of the Main Lemma will occupy Section 5.  We now turn to the proofs of Lemmas \ref{lemma2} and \ref{lemma3}.
 
 \vspace{.15in}
 \noindent\textbf{Notation.}  For a square $B=(i,j)$ in a Ferrers board $F$, we denote $i$ and $j$ by $\col(B)$ and $\row(B)$, respectively.
 
 \vspace{.15in}
 
 \noindent\emph{Proof of Lemma \ref{lemma2}.}  Let $D=D_1D_2\ldots D_k$ be the sequence of squares, from left to right, constituting the smallest occurrence of $k\ldots1$ in $P$. Let $S=S_1S_2\ldots S_k$ be another occurrence of $k\ldots 1$ in $P$.  By the minimality of $D$, we cannot have $\row(S_1)<\row(D_1)$, so $S$ cannot begin in a row below the top row of $R=R(a,b)$, and $S_1=D_1.$  Suppose for a contradiction that $S$ ends in a column to the left of the rightmost column of $R$, so $\col(S_k)<\col(D_k).$
 
 If $\row(S_k)>\row(D_k)$ then $S_2\ldots S_kD_k$ contradicts the minimality of $D$.
 Clearly $\row(S_k)\neq\row(D_k)$, because $S_k$ and $D_k$ are in different columns. So $\row(S_k)<\row(D_k).$ Therefore $S_k\neq D_{k-1}$, so $\col(S_k)\neq\col(D_{k-1}).$  If $\col(S_k)>\col(D_{k-1})$ then $D_1\ldots D_{k-1}S_k$ contradicts the minimality of $D$, since $\row (S_k)<\row(D_k)<\row(D_{k-1}).$ So $\col(S_k)<\col(D_{k-1}).$
 
 Now, assuming inductively that $\col(S_t)< \col(D_{t-1})$ for some $t\geq 3$,  we have $\col(S_{t-1})< \col(D_{t-1})$. We will show by essentially  the same argument as in the preceding paragraph that $\col(S_{t-1})< \col(D_{t-2})$.  First, if $\row(S_{t-1})>\row(D_{t-1})$ then $S_2\ldots S_{t-1}D_{t-1}\ldots D_k$ contradicts the minimality of $D$. So $\row (S_{t-1})< \row(D_{t-1}).$  If $\col(S_{t-1})> \col(D_{t-2})$ then $D_1\ldots D_{t-2}S_{t-1}\ldots S_k$ contradicts the minimality of $D$. Thus $\col(S_{t-1})< \col(D_{t-2}).$
 
 We conclude by induction that $\col(S_2)< \col(D_1)=\col(S_1)$, a contradiction.  $\Box$
 
 \vspace{.15in}
 
 \noindent\emph{Proof of Lemma \ref{lemma3}.}     Again let $D$ be the smallest occurrence of $k\ldots 1$ in $P$, and let $D^*_1,\ldots, D^*_k$ be the squares by which $D_1,\ldots, D_k$ are replaced in $\phi(P)_R.$  Assume for a contradiction that  $\phi(P)_R$ has a $k\ldots 1$-pattern, $S$, in $R$.  Let $c$ be the number of squares that $D^*=D^*_2\dots D^*_k$ has in common with $S$.  (Note that $D_1^*$ is the square in the northeast corner of $R$.) Let $d=k-c$ be the number of squares in $S$ that are not in $D^*.$ Observe that $c\geq 1$.  If this were not the case then $S$ would contain no square $D^*_i$  and therefore it would also be a $k\ldots 1$-pattern in $P_R$.    Then $\row(S_1)<\row (D^*_1)= \row(D_1)$ would contradict the minimality of $D$.  
 
 Now let the $c$ common squares, from left to right,  be $D^*_{m_j}$ for $1\leq j\leq c$. \\
 We group the other squares of $S$  according to their rows.  Let $d_0$ be the number of these other squares that are north of $D^*_{m_1}$ .  Let $d_s$ be the number of the other squares that are south of $D^*_{m_s}$ and north of $D^*_{m_{s+1}}$, for $1\leq s\leq c-1$.  Finally let $d_c$ be the number of the other squares that are south of $D^*_{m_c}$

Note that all the squares counted by $d_c$ must be to the right of $D_{m_c}$, so  we must have  $d_c\leq k-m_c$ because if there were $k-m_c+1$ squares $B_i$ counted by $d_c$ then the sequence  $D_1\ldots D_{m_c-1}$ followed by the $B_i$s would contradict the minimality of $D$.  Next, we must have $d_s\leq m_{s+1}-m_s-1$.  To see this note that the squares counted by $d_s$ must be west of the column containing $D^*_{m_{(s+1)}}$ and east of the column containing $D^*_{m_s}$.  Therefore if there were $m_{s+1}-m_s$ squares $B_i$ counted by $d_s$ the sequence $D_1D_2\ldots D_{m_s-1}$, followed by the $B_i$s and then $D_{m_{s+1}}\ldots D_k$ would contradict the minimality of $D$.   Finally, observe that $d_0\leq m_1-2$.  For if there were $m_1-1$ squares counted by $d_0$ then these $B_i$s followed by $D_{m_1}\ldots D_k$ would contradict the minimality of $D$.

Now when $c>1$ we have
\begin{align*}
k&=c+d \\
&= c+d_0+d_c+\sum_{s=1}^{c-1} d_s\\
&\leq c+(m_1-2)+(k-m_c)+\sum_{s=1}^{c-1}m_{s+1}-m_s-1\\
&\leq (c-1)+(k-1)+m_1-m_c+m_c-m_1-(c-1)\\
&= k-1.
\end{align*}
When $c=1$ we obtain the similar contradiction $$k=1+d_0+d_1\leq 1+(m_1-2)+(k-m_1)=k-1.$$  Therefore $\phi(P)_R$ cannot contain a $k\ldots 1$-pattern in $R$, as claimed.  $\Box$

\vspace{.25in}
\noindent \textbf{4. Pivots}

\vspace{.15in}

In this section we introduce the \emph{left} and \emph{right pivots} of rook placement on a rectangular Ferrers board, and show how they relate to the Robinson-Schensted correspondence and to $\phi^*$.  While the right pivots have nicer properties in connection with the RS correspondence, the left pivots have nicer properties with respect to $\phi^*$.  

\begin{definitions}
Let $P$ be a placement on rectangular Ferrers board $F$. We define the set of \emph{left pivots of $P$} (respectively, \emph{right pivots of $P$}) to be the placement $\pivL\ P$ (respectively, $\pivR\ P$) defined by inductively placing markers, row by row, from bottom to top, as follows.  

First, there is no pivot in the bottom row. Now consider row $r>1$. If there is no element of $P$ in row $r$ then we do not place a pivot in row $r$.  Now suppose $X\in P$ is in row $r$. If there is a column to the left (respectively, right) of $X$ that contains an element of $P$ below row $r$ but does not contain a pivot then we place a pivot in row $r$ and in the rightmost (respectively, leftmost) such column. 
\end{definitions}

\begin{notation}
For a placement $P$ on a rectangular Ferrers board $F$ define $\rev(P)$ to be the placement obtained by reflecting $F$ and $P$ along a vertical line.
\end{notation}

Below is an example of left and right pivots of a placement.

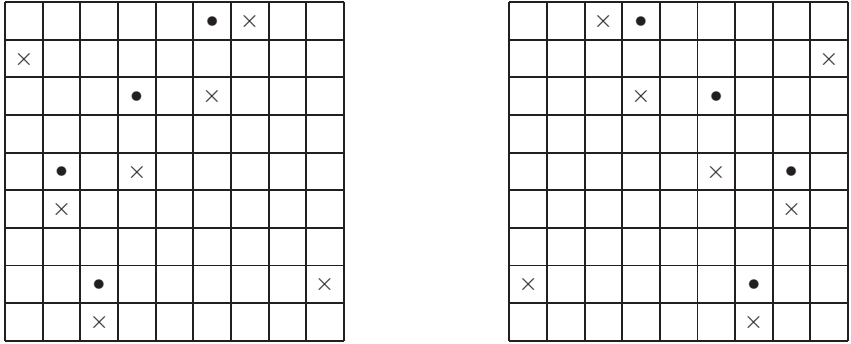
\begin{figure}[h!]
\begin{center}
\ifx\JPicScale\undefined\def\JPicScale{1}\fi
\unitlength \JPicScale mm
$\begin{array}{ccc}

\begin{picture}(60,45)(0,0)
\linethickness{0.1mm}

%Horizontal Lines
\put(0,0){\line(1,0){45}}
\put(0,5){\line(1,0){45}}
\put(0,10){\line(1,0){45}}
\put(0,15){\line(1,0){45}}
\put(0,20){\line(1,0){45}}
\put(0,25){\line(1,0){45}}
\put(0,30){\line(1,0){45}}
\put(0,35){\line(1,0){45}}
\put(0,40){\line(1,0){45}}
\put(0,45){\line(1,0){45}}

%Vertical Lines
\put(0,0){\line(0,1){45}}
\put(5,0){\line(0,1){45}} 
\put(10,0){\line(0,1){45}}
\put(15,0){\line(0,1){45}}
\put(20,0){\line(0,1){45}}
\put(25,0){\line(0,1){45}}
\put(30,0){\line(0,1){45}}
\put(35,0){\line(0,1){45}}
\put(40,0){\line(0,1){45}}
\put(45,0){\line(0,1){45}}

%Placement
\put(42.5,7.5){\makebox(0,0)[cc]{\small{$\times$}}}
\put(32.5,42.5){\makebox(0,0)[cc]{\small{$\times$}}}
\put(27.5,32.5){\makebox(0,0)[cc]{\small{$\times$}}}
\put(17.5,22.5){\makebox(0,0)[cc]{\small{$\times$}}}
\put(12.5,2.5){\makebox(0,0)[cc]{\small{$\times$}}}
\put(7.5,17.5){\makebox(0,0)[cc]{\small{$\times$}}}
\put(2.5,37.5){\makebox(0,0)[cc]{\small{$\times$}}}

%Pivots
\put(12.5,7.5){\makebox(0,0)[cc]{\small{$\bullet$}}}
\put(7.5,22.5){\makebox(0,0)[cc]{\small{$\bullet$}}}
\put(17.5,32.5){\makebox(0,0)[cc]{\small{$\bullet$}}}
\put(27.5,42.5){\makebox(0,0)[cc]{\small{$\bullet$}}}
\end{picture}
&&

\begin{picture}(60,45)(0,0)
\linethickness{0.1mm}

%Horizontal Lines
\put(0,0){\line(1,0)  {45}}
\put(0,5){\line(1,0)  {45}}
\put(0,10){\line(1,0){45}}
\put(0,15){\line(1,0){45}}
\put(0,20){\line(1,0){45}}
\put(0,25){\line(1,0){45}}
\put(0,30){\line(1,0){45}}
\put(0,35){\line(1,0){45}}
\put(0,40){\line(1,0){45}}
\put(0,45){\line(1,0){45}}

%Vertical Lines
\put(0,0){\line(0,1)  {45}}
\put(5,0){\line(0,1)  {45}}
\put(10,0){\line(0,1){45}}
\put(15,0){\line(0,1){45}}
\put(20,0){\line(0,1){45}}
\put(25,0){\line(0,1){45}}
\put(30,0){\line(0,1){45}}
\put(35,0){\line(0,1){45}}
\put(40,0){\line(0,1){45}}
\put(45,0){\line(0,1){45}}

%Placement

\put(2.5,7.5){\makebox(0,0)[cc]{\small{$\times$}}}
\put(12.5,42.5){\makebox(0,0)[cc]{\small{$\times$}}}
\put(17.5,32.5){\makebox(0,0)[cc]{\small{$\times$}}}
\put(27.5,22.5){\makebox(0,0)[cc]{\small{$\times$}}}
\put(32.5,2.5){\makebox(0,0)[cc]{\small{$\times$}}}
\put(37.5,17.5){\makebox(0,0)[cc]{\small{$\times$}}}
\put(42.5,37.5){\makebox(0,0)[cc]{\small{$\times$}}}

%Pivots
\put(32.5,7.5){\makebox(0,0)[cc]{\small{$\bullet$}}}
\put(37.5,22.5){\makebox(0,0)[cc]{\small{$\bullet$}}}
\put(27.5,32.5){\makebox(0,0)[cc]{\small{$\bullet$}}}
\put(17.5,42.5){\makebox(0,0)[cc]{\small{$\bullet$}}}
\end{picture}
\end{array}$
\caption{On the left we have $P$ and $\pivL(P)$.  On the right we have $\rev(P)$ and $\pivR(P)$.  In both examples elements of $P$ are denoted by $\times$ and the pivots are denoted by $\bullet$.}
\end{center}
\end{figure}

Looking at Figure 1 we see the following relationship between left and right pivots. Its proof is straightforward.
\begin{lemma}\label{lem_left_right_relation}
For any placement $P$ on a rectangular Ferrers board $F$ we have
$$\rev\Big(  \pivL\ P\Big) = \pivR\Big( \rev\ P\Big).$$
\end{lemma}

The utility of pivots is due to their connection with the Robinson-Schensted (RS) algorithm.  We will consider applying the RS algorithm to a placement $P$.  What we really mean is that we are implicitly applying it to the partial permutation corresponding to $P$.  We will also write $\ins(P)$ and $\rec(P)$ for the insertion and recording tableaux of $P$.  Let us first compare the insertion and recording tableaux for $P$, $\pivR(P)$, and $\pivL(P)$ using the placement on the right hand side in Figure 1.  

When we apply RS to $P$ we get
\begin{displaymath}
\begin{array}{ccp{.4in}cc}
\raisebox{.8cm}{$\ins\ P =$} &\young(148,25,7,9) && \raisebox{.8cm}{$\rec\ P =$} &\young(139,48,6,7)
\end{array}.
\end{displaymath} 

Now on the one hand RS applied to $\pivR\ P$ gives us

\begin{displaymath}
\begin{array}{ccp{.4in}cc}
\raisebox{.5cm}{$\ins(\pivR\ P) = $}&\young(25,7,9) && \raisebox{.5cm}{$\rec(\pivR\ P)= $}  &\young(48,6,7)
\end{array},
\end{displaymath} 

while RS applied to $\pivL\ P$ gives

\begin{displaymath}
\begin{array}{ccp{.4in}cc}
\raisebox{.5cm}{$\ins(\pivL\ P) = $}&\young(48,5) && \raisebox{.5cm}{$\rec(\pivL\ P)= $}  &\young(18,7)
\end{array}.
\end{displaymath} 

\noindent We see here that the RS algorithm applied to $\pivR\ P$ gives the same insertion and recording tableaux, minus the top row, as the RS algorithm applied to $P$.  And the insertion tableau of $\pivL\ P$ is equal to $\ins\ P$, minus its left column.  This is the content of the next two theorems.

\begin{notation}
If $Y$ is a standard Young tableau then we will denote by $Y^-$ the tableau consisting of all but the top row of $Y$ and $\leftexp{-}{Y}$ the tableau consisting of all but the left column of $Y$.\\
\end{notation}

%Pivot Theorem #1 
\begin{theorem} \label{thm_pivots}
Let $P$ be a placement on a rectangular Ferrers board.  Then we have
$$\ins( \pivR\ P) = (\ins\ P)^-\textrm{ \hspace{.2in}and \hspace{.2in}} \rec( \pivR\ P)   = (\rec\ P)^-$$
\end{theorem}

\begin{proof}  
Let 
\begin{displaymath}
\pi=\left(
\begin{array}{cccc}
i_1 & i_2 & \cdots & i_n \\
s_1 & s_2 & \cdots & s_n
\end{array}
\right),
\end{displaymath}
with $i_1<i_2<\ldots$, be the partial permutation corresponding to $P$.  Let 

\begin{displaymath}
\left(
\begin{array}{cccc}
p_1 & p_2 & \cdots & p_m \\
v_1 & v_2 & \cdots & v_m
\end{array}
\right),
\end{displaymath}
with $p_1<p_2<\ldots$, be the partial permutation corresponding to $\pivR(P)$.  

If we consider the RS algorithm applied to $\pi$ it will suffice to show that inserting $\pi(p_j)$ into the insertion tableau causes $v_j$ to be bumped from the first row and that the $v_j$ are the only numbers bumped from the first row.  To show this we use induction on $m$, the number of pivots.  If $m=0$ then $\pi$ must have no 21-pattern and hence $s_1<s_2<\ldots$.  Hence $(\ins\ \pi)^-=\emptyset$ and $(\rec\ \pi)^-=\emptyset$.  Now assume $\pi$ is such that $m>0$.  Define $k$ so that $i_k=p_m$ and consider $\pi|_{k-1}$, the restriction of $\pi$ to its first $k-1$ entries.
Then we have 

\begin{displaymath}
\pivR(\pi|_{k-1})=\left(
\begin{array}{cccc}
p_1 & p_2 & \cdots & p_{m-1} \\
v_1 & v_2 & \cdots & v_{m-1}
\end{array}
\right).
\end{displaymath}

By induction the RS algorithm applied to $\pi|_{k-1}$ bumps $v_j$ on the $p_j$th move for $1\leq j\leq m-1$ and bumps nothing else.  So the top row of the insertion tableau is just 
$$S=\{s_j\ |\  1\leq j\leq k-1\} \setminus \{v_j\ | \ 1\leq j \leq m-1\}.$$  
Since applying the RS algorithm to $\pi|_{k-1}$ is the same as partially computing the RS algorithm on $\pi$ up through the index $i_{k-1}$ of $\pi$ we may assume this is where we are in the algorithm.   We now claim that inserting the next value of $\pi$, $s_k$, into the insertion tableau bumps $v_m$ from the first row of the insertion tableau.  Showing this is equivalent to showing that $v_m\in S$ and that $v_m$ is the smallest element in $S$ that is greater than $s_k$.  Clearly $v_m\in S$ as all the $v_j$ are distinct.  Now if $s\in \{ s_j\ |\ 1\leq j\leq k-1\}$ and $s_k < s<v_m$ then $ss_k$ is a 21-pattern.  But this implies that for some $i<m$, $s = v_i$.  Hence  $s\notin S$ and $v_m$ is indeed bumped by the insertion of $s_k$.  Therefore after the $k$th move the first row of the insertion tableau is 
$$S^\prime =\{s_j\ |\  1\leq j\leq k\} \setminus \{v_j\ | \ 1\leq j \leq m\}.$$  
Now let $k<j\leq n$. To complete the proof we need only show that the RS algorithm appends $s_j$ to the top row of the insertion tableau. This is equivalent to showing that $s_{k+1}<s_{k+2}<\ldots$ and that $s_{k+1}$ is greater than all the elements in $S^\prime$.    If the $s_j$ are not increasing from $j=k+1$ onward then they must contain a 21-pattern and hence there exists a pivot to the right of position $p_m$, contradicting our choice of $p_m$.  If $s_{k+1}$ is not greater than all the elements in $S'$ then for some $s'\in S'$, $s's_{k+1}$ a 21-pattern.  This means that either $s' = v_i$, for some $i$, or $i_{k+1}= p_j$, for some $j$.  But this contradicts either our definition of $S'$ or our definition of $p_m$.
\end{proof}

\begin{remark}
We thank Sergi Elizalde for suggesting to us that there might be a connection between our right pivots and Viennot's geometric construction.  It turns out that the right pivots coincide with Viennot's ``northeast corners".  This follows from the fact that  Viennot establishes the analogue of Theorem \ref{thm_pivots} for the northest corners.  
\end{remark}

To state our theorem relating left-pivots and the RS algorithm we first need the following.
 
\begin{notation} If $Y$ is any standard Young tableau then denote by $Y^{tr}$ the transposed tableau, i.e., the tableau obtained by reflecting $Y$ across the NW-SE diagonal.

Likewise, if $P$ is any placement on a rectangular Ferrers board $F$ then denote by $F^{tr}$ and $P^{tr}$ the resulting board and placement obtained by reflecting $F$ and $P$ across the NW-SE diagonal.
\end{notation}

\begin{remark}\label{remark_reverse}
Recall the well known fact that for any placement $P$ we have 
$$\ins( \rev\ P) = (\ins\ P)^{tr}.$$
For a proof of this result see [7], page 97.
\end{remark}

%Pivot Theorem #2 
\begin{theorem}\label{thm_left_pivots}
Let $P$ be a placement on a rectangular Ferrers board $F$. Then
$$ \ins(\pivL\ P)  = \leftexp{-}{(\ins\ P)}.$$
\end{theorem}  

\begin{proof}
By Lemma \ref{lem_left_right_relation} we have
$$\rev( \pivL\ P) = \pivR( \rev\ P).$$
Taking the insertion tableau of both sides and applying Theorem \ref{thm_pivots} to the right hand side we get
$$\ins\Big( \rev\ (\pivL\ P)\Big)  = [\ins\Big(\rev\ P\Big)]^-.$$ 
Then by Remark 2 we have 
$$(\ins\ (\pivL\ P))^{tr}=  [(\ins\ P)^{tr}]^-.$$
By then transposing both sides we have
$$\ins(\pivL\ P)=  ([(\ins\ P)^{tr}]^-)^{tr} =  \leftexp{-}{ (\ins\ P)}.$$  
\end{proof}

\begin{definitions}
Let $P$ be a placement on a rectangular Ferrers board.  For any $X\in P$, if $X$ is in the same column as some $V\in \pivL\ P$ then define $\rho(P,X) = \row(V)$ else define $\rho(P,X) = \infty$.  

Likewise, if $X$ is in the same row as some $V \in \pivL\ P$ then define $\kappa(P,X) =\col(V)$ else define $\kappa(P,X) = 0$. 

If the placement is understood we will just write $\rho(X)$ or $\kappa(X)$.
\end{definitions}

Next let us establish some symmetry intrinsic to left-pivots
\begin{lemma}\label{lem_pivot_symmetry}
If $P$ is a placement on a rectangular Ferrers board $F$ with $n$ columns, then we have
$$(\pivL\ P)^{tr} = \pivL( P^{tr})$$
\end{lemma}

\begin{proof}
Observe that this is equivalent to showing that the set $\pivL(P)$ may be constructed by the following column construction.  

\vspace{.15in}
\noindent\emph{Column Construction}

First, there is no pivot in the rightmost column. Now consider column $c<n$. If there is no element of $P$ in column $c$ then we do not place a pivot in column $c$.  If $X\in P$ is in column $c$ then if there is a row above $X$ that contains no pivot to the right of column $c$ but does contain an element of $P$ to the right of column $c$ then there is a pivot in column $c$ and in the lowest such such row. 
\vspace{.15in}

For reference we will call these pivots \emph{column pivots}.  To show that the column and row constructions agree assume inductively that
\begin{equation}\label{col_row_pivot}
\textrm{column pivots$(P) =\pivL(P)$, for all columns to the right of $c$.}
\end{equation}

Assume for a contradiction that the column pivots and left pivots disagree in column $c$.  Note that if there is no column pivot in column $c$ then by (\ref{col_row_pivot}) we cannot have a left pivot either. Now assume we have a column pivot in $(c,r)$, for some $r$.  So there must exist $X,Y\in P$ with $\col(X) =c$, $\row(Y) = r$ and $XY$ a 12-pattern.  If there is no left pivot in $(c,r)$ then it follows from (\ref{col_row_pivot}) that $\kappa(Y)<c$.  But this can only occur if there exists some $Z\in P$ with $\row(Z)<r$ and $\kappa(Z)= c$.  But then (\ref{col_row_pivot}) and the column construction imply that the column pivot in column $c$ would be in some row $<r$ and not in row $r$ as assumed.
\end{proof}
The following definition and lemmas will be useful in Sections 5 and 6. 

\vspace{.15in}
\noindent\textbf{Convention.}  Since we will only be working with left pivots for the remainder of this paper, the term pivot will mean left pivot from now on.
\vspace{.15in}

\begin{definition}
We say an increasing subsequence $I$ of $P$ is a \emph{pivot-path}, if for all $i<|I|$, $\rho(I_i) = \row(I_{i+1})$, i.e., each consecutive pair creates a pivot.
\end{definition}

%Pivot path Figure------------------------------------------------------
\begin{figure}[h!]
\begin{center}
\ifx\JPicScale\undefined\def\JPicScale{.75}\fi
\unitlength \JPicScale mm
\begin{picture}(70,65)(0,0)
\linethickness{0.1mm}

%Horizontal Lines
\put(0,0){\line(1,0){85}}
\put(0,5){\line(1,0){85}}
\put(0,10){\line(1,0){85}}
\put(0,15){\line(1,0){85}}
\put(0,20){\line(1,0){85}}
\put(0,25){\line(1,0){85}}
\put(0,30){\line(1,0){85}}
\put(0,35){\line(1,0){85}}
\put(0,40){\line(1,0){85}}
\put(0,45){\line(1,0){85}}
\put(0,50){\line(1,0){85}}
\put(0,55){\line(1,0){85}}
\put(0,60){\line(1,0){85}}
\put(0,65){\line(1,0){85}}
\put(0,70){\line(1,0){85}}

%Vertical Lines
\put(0,0){\line(0,1){70}}
\put(5,0){\line(0,1){70}}
\put(10,0){\line(0,1){70}}
\put(15,0){\line(0,1){70}}
\put(20,0){\line(0,1){70}}
\put(25,0){\line(0,1){70}}
\put(30,0){\line(0,1){70}}
\put(35,0){\line(0,1){70}}
\put(40,0){\line(0,1){70}}
\put(45,0){\line(0,1){70}}
\put(50,0){\line(0,1){70}}
\put(55,0){\line(0,1){70}}
\put(60,0){\line(0,1){70}}
\put(65,0){\line(0,1){70}}
\put(70,0){\line(0,1){70}}
\put(75,0){\line(0,1){70}}
\put(80,0){\line(0,1){70}}
\put(85,0){\line(0,1){70}}

%Markers 
\put(2.5,37.5){\makebox(0,0)[cc]{ \small{$\times$} }}
\put(17.5,27.5){\makebox(0,0)[cc]{ \small{$\times$} }}
\put(27.5,52.5){\makebox(0,0)[cc]{ \small{$\times$} }}
\put(32.5,42.5){\makebox(0,0)[cc]{ \small{$\times$} }}
\put(37.5,12.5){\makebox(0,0)[cc]{ \small{$\times$} }}
\put(42.5,57.5){\makebox(0,0)[cc]{ \small{$\times$} }}
\put(47.5,47.5){\makebox(0,0)[cc]{ \small{$\times$} }}
\put(52.5,67.5){\makebox(0,0)[cc]{ \small{$\times$} }}
\put(57.5,22.5){\makebox(0,0)[cc]{ \small{$\times$} }}
\put(67.5,32.5){\makebox(0,0)[cc]{ \small{$\times$} }}
\put(72.5,7.5){\makebox(0,0)[cc]{ \small{$\times$} }}
\put(82.5,17.5){\makebox(0,0)[cc]{ \small{$\times$} }}

%Pivots
\put(72.5,17.5){\makebox(0,0)[cc]{ \tiny{$\bullet$} }}
\put(57.5,32.5){\makebox(0,0)[cc]{ \tiny{$\bullet$} }}
\put(47.5,67.5){\makebox(0,0)[cc]{ \tiny{$\bullet$} }}
\put(37.5,22.5){\makebox(0,0)[cc]{ \tiny{$\bullet$} }}
\put(32.5,47.5){\makebox(0,0)[cc]{ \tiny{$\bullet$} }}
\put(27.5,57.5){\makebox(0,0)[cc]{ \tiny{$\bullet$} }}
\put(17.5,42.5){\makebox(0,0)[cc]{ \tiny{$\bullet$} }}
\put(2.5,52.5){\makebox(0,0)[cc]{ \tiny{$\bullet$} }}

\color{red}
\linethickness{0.2mm}
\put(74,17.5){\line(1,0){7}}
\put(72.5,9){\line(0,1){7}}

\put(37.5,14){\line(0,1){7}}
\put(39,22.5){\line(1,0){17}}
\put(57.5,24){\line(0,1){7}}
\put(59,32.5){\line(1,0){7}}

\put(17.5,29){\line(0,1){12}}
\put(19,42.5){\line(1,0){12}}
\put(32.5,44){\line(0,1){2}}
\put(34,47.5){\line(1,0){12}}
\put(47.5,49){\line(0,1){17}}
\put(49,67.5){\line(1,0){2}}

\put(29,57.5){\line(1,0){12}}
\put(27.5,54){\line(0,1){2}}
\put(4,52.5){\line(1,0){22}}
\put(2.5,39){\line(0,1){12}}
\end{picture}
\caption{An example of pivot-paths where the placement is denoted by $\times$ and the pivots are denoted by $\bullet$.  The consecutive elements of each pivot path are connected by red lines.}
\end{center}
\end{figure}
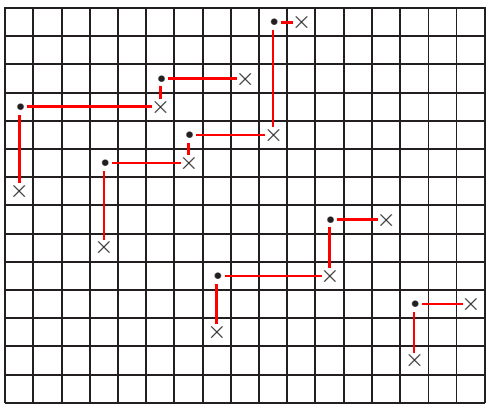

\newpage
\begin{lemma}\label{lemma_pivot1}
If $X, Y \in P$ are such that $XY$  is a 12-pattern then we cannot have $\row(Y)<\rho(X)$ and $\kappa(Y)<\col(X)$.
\end{lemma}

\begin{proof}
If both of these held, then according to the construction of pivots we would have $\col(X)\leq \kappa(Y)$.  
\end{proof}

\begin{lemma}\label{lemma_pivot2}
Let $X,Y,Z\in P$ be such that $XYZ$ is a 132-pattern.  If there is a pivot-path $I$ with first element $X$ and last element $Y$ then there exists a pivot-path $J$ with last element $Z$ such that $XJ_1$ is a 21-pattern.
\end{lemma}

\begin{proof}
Let $K$ be the longest pivot-path with last element $Z$. Note that $K\cap I=\emptyset$, for if not then $K\cup I$ is a pivot-path but this is prohibited as $XYZ$ is a 132-pattern.  Observe that $I$ and $K$ cannot ``cross", i.e., for all 12-patterns $I_iK_j$, we cannot have $\row(K_j)<\rho(I_i)$ and $\kappa(K_j)<\col(I_i)$ by Lemma \ref{lemma_pivot1}.  Therefore, if $XK_1$ is a 12-pattern then for some $i$, $I_iK_1$ is a 12-pattern with $\row(K_1)<\rho(I_i)$.  But by the maximality of the length of $K$ we must have $\kappa(K_1)=0<\col(I_i)$.  But this cannot happen by Lemma $\ref{lemma_pivot1}$.  Therefore, if $K_1$ is to the right of $X$ we can take $J=K$.  If $K_1$ is to the left of $X$, define $K_l$ to be the first element of $K$ that is to the right of $X$.  Observe that $K_l$ must be below $X$ since $J$ and $I$ do not cross.  Take $J = K_l K_{l+1}\ldots Z$.
\end{proof}

The statement of the next lemma is similar to Lemma \ref{lemma_pivot2}.  Its proof is also similar to that of Lemma \ref{lemma_pivot2}, but simpler.

\begin{lemma}\label{lemma_pivot3}
Let $X,Y\in P$ such that $XY$ is a 12-pattern.  If $\row(Y)<\rho(X)$ then there exists a pivot-path $J$, with last element $Y$, such that $XJ_1$ is a 21-pattern.
\end{lemma}

\begin{lemma}\label{lemma_pivot4}
Let $S$ be a decreasing sequence of length $k$ in $P$.  Assume there is some $X\in P$ such that $XS_1$ is a 12-pattern.  If $\rho(X)>\row(S_1)$ there exists a decreasing sequence $D$ of length $k+1$ with $D_1=X$, and if  $X$ and $S_1$ are the first and last elements of a pivot-path then there exists a decreasing sequence $D$ of length $k$ with $D_1=X$.
\end{lemma}

\begin{proof}
First consider the case when $X$ and $S_1$ are connected by a pivot-path.  We will proceed by induction on $k$, the length of $S$.  If $k=1$ then we may take $D=X$.  Now consider $k>1$.  If $\row(X)>\row(S_2)$ then we may take $D=XS_2\ldots S_k$.  If on the other hand $\row(X)<\row(S_2)$ then Lemma \ref{lemma_pivot2} implies that there exists a pivot-path $J$ with last element $S_2$ and $XJ_1$ a 21-pattern.  Now by the induction hypothesis applied to $J_1$ and the shorter sequence $S_2S_3\ldots S_k$ we have a decreasing sequence $E$ of length $k-1$ with $E_1 = J_1$. But then the decreasing sequence $D=XE_1E_2\ldots E_{k-1}$ is what we want.

For the case when $\rho(X)>\row(S_1)$ first apply Lemma \ref{lemma_pivot3} to the 12-pattern $XS_1$ to obtain a pivot-path $J$ such than  $XJ_1$ forms a 21-pattern and $J$'s last element is $S_1$.  Now apply the first case to $J_1$ and the sequence $S$ to obtain a decreasing sequence $E$ of length $k$ with $E_1 = J_1$.  Then we may take $D = XE_1E_2\ldots E_k$ which is of length $k+1$. 
\end{proof}

Note that the previous lemma and Lemma \ref{lem_pivot_symmetry} directly imply the following result.

\begin{lemma}\label{lemma_pivot5}
Let $S$ be a decreasing sequence of length $k$ in $P$.  Assume there is some $X\in P$ such that $S_1X$ is a 12-pattern.  If $\kappa(X)<\col(S_1)$ then there exists a decreasing sequence $D$ of length $k+1$ with $D_1=X$.
\end{lemma}
%Section 5--5--5--5--5--5--5--5--5--5--5--5--5--5--5--5--5--5--5--5--5--5--5--5--5--5--5--5--5--5--5
\vspace{.25in}
\noindent \textbf{5. The Proof of the Main Lemma}
\vspace{.25in}

The purpose of this section is to prove the Main Lemma which is the crucial piece needed to show that our growth diagram construction corresponds to the map $\phi^*$.  Note that up through Lemma \ref{lemma_same_insertion}, $F$ will always denote a rectangular Ferrers board and $P$ a rook placement on $F$. 

We begin with some notation and definitions. 

\begin{notation}
Define $F|_{a,b}$ to be the Ferrers board consisting of all columns of $F$ between and including columns $a$ and $b$, and define $P|_{a,b}$ to be the restriction of $P$ to $F|_{a,b}$. 
\end{notation}

\begin{definitions}
Fix a subplacement $S$, let $k=|S|$, and order $S$ such that $S = S_1\ldots S_k$ where $\col(S_i)<\col(S_{i+1})$.  Further fix a column $c$ and assume that $P$ has no element in column $c$. 

We first define a \emph{left-shift}.  Assume $S$ lies entirely to the right of column $c$. Place markers in all the squares of $P$, and then shift each marker in square $S_i$ for $1<i\leq k$ horizontally left to column $\col(S_{i-1})$ and shift the marker in $S_1$ horizontally left to column c.  Define $P(c\leftarrow S)$ to be the squares that now contain markers. Likewise, define $c\leftarrow S$ to be $P(S\leftarrow c)\setminus P$, the placement obtained from $S$ by shifting it left.

Analogously, we define a \emph{right-shift}. Assume $S$ lies entirely to the left of column $c$.   Place markers in all the squares of $P$, and then shift each marker in square $S_i$ for $1\leq i< k$ horizontally right to column $\col(S_{i+1})$ and shift the marker in $S_k$ horizontally right to column c.  Define $P(S\rightarrow c)$ to be the squares that now contain markers. Likewise, define $S\rightarrow c$ to be $P(S\rightarrow c)\setminus P$, the placement obtained from $S$ by shifting it right.
\end{definitions}

\begin{definitions}
Let $P$ be a placement on a rectangular Ferrers board $F$.  Let $a<b$ be columns of $F$.

Let $k$ be the length of the longest decreasing sequence in $P|_{a,b}$. Define $\ds{a,b}(P)$  (respectively, $\DL{a,b}(P)$) to be the smallest (respectively, largest) decreasing sequence, under the lexicographical ordering,  in $P|_{a,b}$ of length $k$.

Likewise, let $m$ be the length of the longest increasing sequence in $P|_{a,b}$. Define $\is{a,b}(P)$  (respectively, $\IL{a,b}(P)$) to be the smallest (respectively, largest) increasing sequence, under the lexicographical ordering, in $P|_{a,b}$ of length $m$.

\end{definitions}

We are now ready to state and prove Lemma \ref{lemma_same_insertion}.  Although Lemma \ref{lemma_same_insertion} is the key driver behind the Main Lemma, its statement is more general then what is needed to prove the Main Lemma.  The reason for this is that the precise statement of Lemma \ref{lemma_same_insertion} is exactly what is needed in Section 6.

\begin{lemma}\label{lemma_same_insertion}
Let $P$ be a placement on a rectangular Ferrers board $F$ and assume $P$ has no marker in column $a$.  Let $S=\ds{a,b}(P)$ and $P^*= P(a\leftarrow S)$.  Then we have
$$\ins(P) = \ins(P^*).$$
\end{lemma}

\begin{proof}
First let $S^*=a\leftarrow S$, i.e., the shifted sequence, and $k=|S|$.  Define $b'$ to be the column containing $S_k$ and consider the truncated board and placements $G = F|_{1,b'}$, $Q=P|_{1,b'}$ and $Q^* = P^*|_{1,b'}$.  Now in order to show  
$$\ins(P) = \ins(P^*),$$
it will suffice to show 
\begin{align}\label{eqn_ins}
\ins(Q) &= \ins(Q^*).
\end{align}
To see this consider the GDA applied to $P$ and $P^*$ on the full board $F$.  By [11], Theorem 7.13.5, knowing (\ref{eqn_ins}) is the same as knowing that the partitions along the line $x=b'$ in the GDA of $P$ and in the GDA of $P^*$ are identical.  But the placements $P$ and $P^*$ are identical east of the line $x=b'$.  Therefore if (\ref{eqn_ins}) holds we know that the partitions along the right border of $F$ in the GDA of $P$ and in the GDA of $P^*$ are also identical. By [11], Theorem 7.13.5, this implies that 
$$\ins(P) = \ins(P^*).$$

In order to show (\ref{eqn_ins}) it is sufficient to prove 
\begin{align}\label{eqn_piv}
\pivL\ Q = \pivL\ Q^*.
\end{align}
To see why note that (\ref{eqn_piv}) along with Theorem \ref{thm_left_pivots} implies that 
$$\leftexp{-}{[\ins(Q)]} = \leftexp{-}{[\ins(Q^*)]}$$
But this forces $ \ins(Q)  =  \ins(Q^*) $ as needed.  

To establish (\ref{eqn_piv}) we will proceed inductively by assuming that the pivots (if any) below some row $r$ are unchanged and then showing that the pivot (if any) on row $r$ is unchanged.  To do so we consider two cases.

\vspace{.15in}

{\tt{Case 1:}} Row $r$ contains an element $S_i$ of $S$.

\vspace{.15in}  

Let $R$ be the rectangular region of $F$ containing all squares west of $\col\ S_i$ and south of $\row\ S_i$.  By the nature of a left-shift, $Q|_R = Q^*|_R$.     Further, by our induction hypothesis $\pivL( Q|_R) = \pivL( Q^*|_R)$.  It then follows that the only way the pivot (if any) on row $r$ could change is if $\pivL( Q)$ contains a pivot in row $r$ between columns $a'$ and $\col(S_i)$, where $a'=a$, if $i=1$, and $a'=\col(S_{i-1})$, if $i>1$.

To show this is impossible assume, for a contradiction, that $\pivL(Q)$ does contain a pivot in row $r$ between columns $a'$ and $\col(S_i)$. Then there must exist some $X\in Q$ that is directly below this pivot.  Since $XS_i$ is a pivot-path, Lemma \ref{lemma_pivot4}, applied to $X$ and $S_iS_{i+1}\ldots $,  implies the existence of a decreasing sequence $D$ of length $k-i+1$ with first element $X$. But this is a contradiction since the decreasing sequence $S_1\ldots S_{i-1} D$ is smaller than $S$.

\vspace{.15in}
{\tt{Case 2:}}  Row $r$  contains $X\in  Q\setminus S$.

\vspace{.15in} 

By the maximality of the length of $S$, $S_{i-1}XS_i$ cannot be a 321-pattern for any $i$.  Therefore it will suffice to assume that $S_iX$ is a 12-pattern for some $i$ and that $i$ is chosen as small as possible. Then $\col(S_i)\leq\kappa(X,Q)$.  If not then Lemma \ref{lemma_pivot5} would give rise to sequence of length $k+1$ contradicting the maximality of the length of $S$.  Therefore we must have an element $Y\in Q$ with $\col(S_i)\leq\col(Y) = \kappa(X,Q)$.  Since $\col(S_i)\leq\col(Y) < b'$ then column $\col(Y)$ must contain an element of $Q^*$ below row $r$. (Note that for the case where $Y\in S$, $Y$ cannot be the last element of $S$.) This plus the induction hypothesis implies that $\col(Y)\leq\kappa(X,Q^*)$.  Assume now, for a contradiction, that $\col(Y)<\kappa(X,Q^*)$ and let $Z\in Q^*$ be such that $\col(Y)<\col(Z)=\kappa(X,Q^*)$.  Now, by similar reasoning, $\col(Z)$ must contain an element of $Q$ below row $r$.  But this plus the induction hypothesis implies that $\kappa(X,Q)=\col(Y)<\col(Z)\leq\kappa(X,Q)$, an obvious contradiction.
\end{proof}

To prove the Main Lemma it suffices, by [11], Theorem 7.13.5, to prove the following.

\begin{lemma} \label{lem_technical_main_lemma}
Let $P$ be a placement on a not necessarily rectangular Ferrers board.  Let $S$ be the smallest $k\ldots 1$-pattern in $P$ and let $b=\row(S_1)$ and $a=\col(S_k)$. If $R_1 = R(a,b-1)$ then 
\begin{align}\label{y}
\ins(P|_{R_1}) &= \ins(\phi(P)|_{R_1}).
\end{align}  
Likewise, if $R_2 = R(a-1,b)$ then
\begin{align}\label{x}
\rec(P|_{R_2})&=\rec( \phi(P)|_{R_2}).
\end{align}
\end{lemma}

\begin{proof}
Let $Q = P|_{R_1}$ and $T = \phi(P)|_{R_1}$.  Observe that $S_2,\ldots,S_k = \ds{\col( S_1),a}(Q)$. Now Lemma \ref{lemma_same_insertion} implies (\ref{y}). 

For (\ref{x}) let $R=R(a,b)$ and $Q = P|_R$ and $T = \phi(P)|_R$.  Observe that  $T' = \phi(Q')$, where $Q'$ and $T'$ are the inverses of $Q$ and $T$ in the sense of Section 2.   This follows from the fact that $S$ is the decreasing sequence of maximal length in $R$, which implies that $S$ is also the position smallest sequence of length $k$ in $R$ (i.e., $S$ is the smallest if we order $k\ldots 1$-patterns lexicographically according to the positions of the entries, rather than the values of the entries), so that $S'$ is the value smallest decreasing sequence of length $k$ in $R'$.

\begin{figure}[h!]
\begin{center}
$\begin{array}{cp{.3in}c}
\ifx\JPicScale\undefined\def\JPicScale{1}\fi
\unitlength \JPicScale mm
\begin{picture}(60,50)(0,0)

\linethickness{0.1mm}
\put(0,0){\line(1,0){60}}
\put(0,5){\line(1,0){60}}
\put(0,10){\line(1,0){60}}
\put(0,15){\line(1,0){60}}
\put(0,20){\line(1,0){60}}
\put(0,25){\line(1,0){60}}
\put(0,30){\line(1,0){60}}
\put(0,35){\line(1,0){60}}
\put(0,40){\line(1,0){60}}
\put(0,45){\line(1,0){60}}

\put(60,0){\line(0,1){45}}
\put(55,0){\line(0,1){45}}
\put(50,0){\line(0,1){45}}
\put(45,0){\line(0,1){45}}
\put(40,0){\line(0,1){45}}
\put(35,0){\line(0,1){45}}
\put(30,0){\line(0,1){45}}
\put(25,0){\line(0,1){45}}
\put(20,0){\line(0,1){45}}
\put(15,0){\line(0,1){45}}
\put(10,0){\line(0,1){45}}
\put(5,0){\line(0,1){45}}
\put(0,0){\line(0,1){45}}

%Placement
\put(7.5,42.5){\makebox(0,0)[cc]{$\times$}}
\put(12.5,32.5){\makebox(0,0)[cc]{$\times$}}
\put(17.5,2.5){\makebox(0,0)[cc]{$\times$}}
\put(32.5,22.5){\makebox(0,0)[cc]{$\times$}}
\put(47.5,37.5){\makebox(0,0)[cc]{$\times$}}
\put(57.5,12.5){\makebox(0,0)[cc]{$\times$}}

%Shifted Placement
\color{red}
\put(7.5,32.5){\makebox(0,0)[cc]{$\times$}}
\put(12.5,22.5){\makebox(0,0)[cc]{$\times$}}

\put(32.5,12.5){\makebox(0,0)[cc]{$\times$}}
\put(57.5,42.5){\makebox(0,0)[cc]{$\times$}}

\end{picture}&&\ifx\JPicScale\undefined\def\JPicScale{1}\fi
\unitlength \JPicScale mm
\begin{picture}(45,60)(0,0)

\linethickness{0.1mm}
%Horizontal
\put(0,0){\line(0,1){60}}
\put(5,0){\line(0,1){60}}
\put(10,0){\line(0,1){60}}
\put(15,0){\line(0,1){60}}
\put(20,0){\line(0,1){60}}
\put(25,0){\line(0,1){60}}
\put(30,0){\line(0,1){60}}
\put(35,0){\line(0,1){60}}
\put(40,0){\line(0,1){60}}
\put(45,0){\line(0,1){60}}

%Vertical
\put(0,60){\line(1,0){45}}
\put(0,55){\line(1,0){45}}
\put(0,50){\line(1,0){45}}
\put(0,45){\line(1,0){45}}
\put(0,40){\line(1,0){45}}
\put(0,35){\line(1,0){45}}
\put(0,30){\line(1,0){45}}
\put(0,25){\line(1,0){45}}
\put(0,20){\line(1,0){45}}
\put(0,15){\line(1,0){45}}
\put(0,10){\line(1,0){45}}
\put(0,5){\line(1,0){45}}
\put(0,0){\line(1,0){45}}

%Placement
\put(2.5,17.5){\makebox(0,0)[cc]{$\times$}}
\put(12.5,57.5){\makebox(0,0)[cc]{$\times$}}
\put(22.5,32.5){\makebox(0,0)[cc]{$\times$}}
\put(32.5,12.5){\makebox(0,0)[cc]{$\times$}}
\put(37.5,47.5){\makebox(0,0)[cc]{$\times$}}
\put(42.5,7.5){\makebox(0,0)[cc]{$\times$}}

%Shifted Placement
\color{red}
\put(32.5,7.5){\makebox(0,0)[cc]{$\times$}}
\put(22.5,12.5){\makebox(0,0)[cc]{$\times$}}

\put(12.5,32.5){\makebox(0,0)[cc]{$\times$}}
\put(42.5,57.5){\makebox(0,0)[cc]{$\times$}}

\end{picture}
\end{array}$
\caption{On the left we have an example of $Q$ and $T$.  On the right we have $Q'$ and $T'$.  In both pictures the $\phi(S)$ is marked by the red $\times$'s.}
\end{center}
\end{figure}
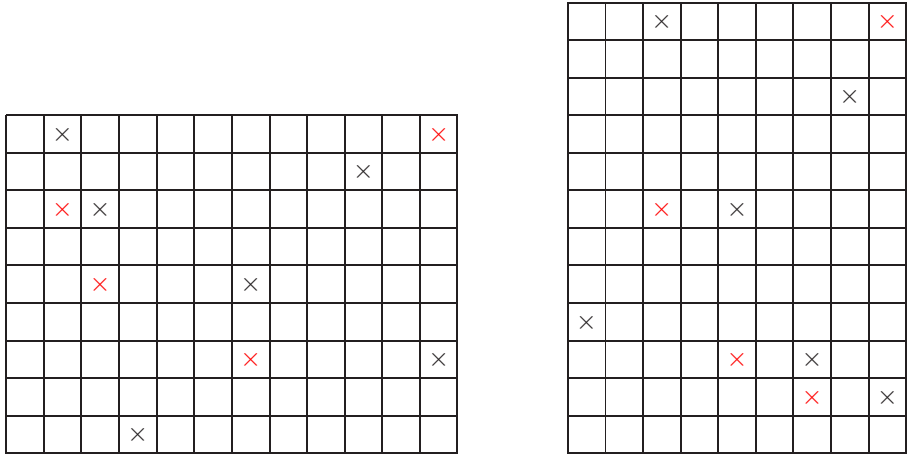
\newpage
Applying (\ref{y}), we know that 
$$\ins( Q'|_{R_2'}) = \ins( T'|_{R_2'}).$$
So by the theorem of Sch$\ddot{\textrm{u}}$tzenberger mentioned in Subsection 2.1 we have
$$\rec(Q|_{R_2})= \rec(T|_{R_2}).$$ 
This completes the proof.  \end{proof}

\vspace{.25in}
%Section 6---6---6---6---6---6---6---6---6---6---6---6---6---6---6---6---6---6---6---6---6---6
\noindent \textbf{6. Generalized Knuth Transformations}
\vspace{.15in}

Recall (see, for example, [11], page 414) that a \emph{Knuth transformation} of a permutation $\sigma$ interchanges two adjacent elements $x$ and $z$ of $\sigma$ provided that there exists a third element $y$, which is adjacent to either $x$ or $z$, such that $x<y<z$.  Two permutations $\sigma$ and $\rho$ are called \emph{Knuth-equivalent}, written $\sigma\sim_K \rho$, if $\rho$ may be obtained from $\sigma$ by a sequence of Knuth transformations.

It is a well known result that two permutations are Knuth-equivalent if and only if they share the same insertion tableau. For a proof of this result see, for example, [11], page 414.    

Fix a placement $P$ on a rectangular Ferrers board $F$.  Further fix columns $a<b$ in $F$.  Let $\ds{} = \ds{a,b}(P)$, $\DL{} =\DL{a,b}(P)$, $\is{} =\is{a,b}(P)$, and $\IL{} =\IL{a,b}(P)$. 

\begin{definition}
Transforming the placement $P$ into either of
\begin{center}
$\begin{array}{ccc}
P(a\leftarrow \ds{})&\textrm{or}&P(a\leftarrow \IL{})\\
\end{array}$
\end{center}
if $P$ has no placement in column $a$, or to either of
\begin{center}
$\begin{array}{ccc}
P(\DL{}\rightarrow b) &\textrm{or}& P(\is{}\rightarrow b)\\
\end{array}$
\end{center}
if $P$ has no placement in column $b$, is called a \emph{generalized Knuth transformation}.
\end{definition}

The next theorem, which is the main point of this section, shows that any two placements that differ by these generalized Knuth transformations are actually Knuth-equivalent.

\begin{theorem}\label{GKE}
If two placements $P$ and $Q$ differ by  generalized Knuth transformations then $\ins(P) = \ins(Q)$.
\end{theorem}

Before launching into the proof of this theorem let us pause to see how these four new transformations are generalizations of the standard Knuth transformations.  First consider the Knuth transformation
$$\underbrace{5\ 3\ 4\ 7\ 6\ 2\ \textbf{8}\ \textbf{1}}_\sigma  \mapsto \underbrace{5\ 3\ 4\ 7\ 6\ 2\ 1\ 8}_\rho.$$
Note that this is of the form $yzx\mapsto yxz$.  Now consider $P_\sigma$ to be the placement corresponding to $\sigma$ on a $9\times 8$ rectangular board where we have an empty column between the `6' and the `2'.   Observe that $\ds{6,9}(P_\sigma)$ corresponds to the subsequence 21 in $\sigma$ and that the partial permutation associated to
$$P_\sigma(6\leftarrow \ds{6,9}(P_\sigma))$$
is order-isomorphic to $\rho$.  It follows that any Knuth transformation of the form $yzx\mapsto yxz$ is generalized by the transformation $P\mapsto P(a\leftarrow \ds{})$.  Similar reasoning demonstrates how the remaining 3 standard Knuth transformations correspond by the remaining 3 generalized Knuth transformations.  For completeness we give the correspondence in the table below.
\begin{center}
$\begin{array}{lp{.5in}c}
\textrm{Generalized} && \textrm{Standard}\\
\hline
P\mapsto P(a\leftarrow \ds{})&&\cdots yzx\cdots\mapsto \cdots yxz\cdots\\
P\mapsto P(\DL{} \rightarrow b)&&\cdots yxz\cdots\mapsto \cdots yzx\cdots\\
P\mapsto P(\is{} \rightarrow b)&&\cdots xzy\cdots\mapsto \cdots zxy\cdots\\
P\mapsto P(a\leftarrow \IL{})&&\cdots zxy\cdots\mapsto \cdots xzy\cdots\\
\end{array}$
\end{center}
Observe that the standard Knuth transformations are clearly reversible. For example $yzx\mapsto  yxz$ and $yxz\mapsto  yzx$ are obviously inverses.  This nice property of the standard transformations extends to the generalized transformations.  We will see below in Lemma \ref{lemma_dec_reverse} that transformations of the forms $P(a\leftarrow \ds{})$ and
$P(\DL{} \rightarrow b)$ are inverses and it follows by considering the reverse of a placement that transformations of the forms $P(\is{} \rightarrow b)$ and $P(a\leftarrow \IL{})$ are inverses.

It remains to prove Theorem \ref{GKE}.  This will occupy the remainder of this section.  We start with a simple definition.

\begin{definition}
If $S$ is a decreasing sequence in $P$ then we say that a subplacement $A\subset P$ is \emph{above} $S$ if for every $C\in A$ there exists some index $i$ such that $S_iC$ is a 12-pattern.  Likewise, we say that $B\subset P$ is \emph{below} $S$ if for every $C\in B$ there exists some $i$ such that $CS_i$ is a 12-pattern.
\end{definition}

\begin{lemma}\label{lemma_shift_right_+1}
Let $L$ be a longest decreasing sequence in $P|_{a,b}$, of length $k$.  If $L\neq \DL{}=\DL{a,b}(P)$ then $P(L\rightarrow b)|_{a,b}$ contains a decreasing sequence of length $k+1$.  
If $L \neq \ds{}$ then $P(a\leftarrow L)|_{a,b}$ contains a decreasing sequence of length $k+1$.  
\end{lemma}

\begin{proof}
We prove only the first assertion.  The second is analogous.  

As $L\neq \DL{}$ we must have some sequence $D$ in $F|_{a,b}$ with $|D|=k$ that is value-larger than $L$.  For $D$ to be value-larger than $L$ we must have some $0\leq l$ and some $l+1\leq m$ such that 
$$D = L_1\ldots L_l D_{l+1}\ldots D_mD_{m+1}\ldots D_k$$
where $D_{l+1}\ldots D_m$ is above $L$ and $D_{m+1}$ is not above $L$, or $D_m$ is the last element of $D$.  Now define $n$ to be the index such that $D_m$ is between the columns containing $L_{n}$ and $L_{n+1}$, or $n=k$ if $D_m$ is the last element of $D$. 

Clearly $D_{m+1}\neq L_{n}$ and as $D_{m+1}$ is not above $L$ then $L_{n}D_{m+1}$ must be a 21-pattern. Therefore
$$L_1\ldots L_{n}D_{m+1}\ldots D_k$$
is a decreasing sequence in $P|_{a,b}$.  Since no decreasing sequence in $P|_{a,b}$ has length greater than $k$ then we must have $n + k-m\leq k$, i.e., $n\leq m$.    Further, observe that
$$D_1\ldots D_m L_{n+1}\ldots L_k$$
is also a decreasing sequence in $P|_{a,b}$.  By similar logic we must have $m\leq n$. Therefore $m=n$. 

Now consider the right-shift and let $L^* = L\rightarrow b$.  Finally, observe that
$$L_1^*\ldots L_l^* D_{l+1}\ldots D_m L^*_{n} \ldots L^*_k$$
is a decreasing sequence in $P(L\rightarrow b)|_{a,b}$.  Since $m=n$ its length is $k+1$ as claimed.  
\end{proof}

\begin{lemma}\label{lemma_SL_left_shift}
We have the following relationships for decreasing sequences:
\begin{center}
\begin{align*}
a\leftarrow \ds{}&\textrm{ is longest in }P(a\leftarrow \ds{})|_{a,b} \\
\DL{}\rightarrow b&\textrm{ is longest in }P(\DL{}\rightarrow b)|_{a,b}
\end{align*}
\end{center}
\end{lemma}

\begin{proof}
First let $Q = P|_{a,b}$ and define $n = b-a+1$, the number of columns in $Q$.  Now define 
\begin{center}
$\begin{array}{lp{.5in}l}
Q_1 = Q(1\leftarrow \ds{1,n}(Q))&&Q_2 =Q(\DL{1,n}(Q)\rightarrow n). 
\end{array}$
\end{center}
It will suffice, by the result of Schensted mentioned in Subsection 2.1, to prove that 
$$ \shape(Q_i) = \shape (Q).$$  

For the first assertion, $\shape(Q_1) = \shape(Q)$ follows directly from Lemma \ref{lemma_same_insertion}.

For the second assertion let $Q^*$ and $Q_2^*$ be the complements of the reverses of $Q$ and $Q_2$.  (By the complement of $Q$ we mean the placement obtained by flipping $Q$ and $F|_{a,b}$ about a horizontal line.) Note that
$$Q_2^* = Q^*(1\leftarrow \ds{1,n}(Q^*)).$$
From the first assertion we know that $\shape(Q_2^*) = \shape(Q^*)$.  Therefore $\shape(Q_2) = \shape(Q)$.
\end{proof}

The next lemma will enable us to reverse  generalized Knuth transformations.  We state the results explicitly only for decreasing sequences.  As remarked previously the results for increasing sequences follow immediately.  

\begin{lemma}\label{lemma_dec_reverse}
Let $k = |\ds{}|$ and assume that $\mathbf{d}_k$ is in column $b$.  Then we have
$$a\leftarrow \ds{} = \DL{a,b}(P(a\leftarrow \ds{})).$$
Likewise, if $\mathbf{D}_1$ is in column $a$,  then we have
$$\DL{}\rightarrow b = \ds{a,b}(P(\DL{}\rightarrow b)).$$
\end{lemma}

\begin{proof}
For the first claim set $P^* = P(a\leftarrow \ds{})$. Assume for a contradiction that $a\leftarrow \ds{}\neq  \DL{a,b}(P^*)$.  By Lemma \ref{lemma_SL_left_shift} we know that $a\leftarrow \ds{}$ is a longest sequence in $P(a\leftarrow \ds{})|_{a,b}$.  Now Lemma \ref{lemma_shift_right_+1} implies that  $ P^*((a\leftarrow \ds{})\rightarrow b)|_{a,b} =P|_{a,b}$ has a decreasing subsequence of length $k+1$.  But this is impossible since $|\ds{}| = k$.  

A similar argument may be used to obtain the second claim.  The details are straightforward. 
\end{proof}

\begin{proof}[Proof of Theorem \ref{GKE}.]
It will suffice to prove that the insertion tableau of a placement is invariant under any of the four generalized Knuth transformations.  We give the arguments for decreasing sequences only. 

\vspace{.15in}
\noindent
{\tt{Case 1:}} $P(a\leftarrow \ds{}).$
\vspace{.15in}

This is clear by \ref{lemma_same_insertion}.

\vspace{.15in}
\noindent
{\tt{Case 2:}} $P(\DL{}\rightarrow b).$
\vspace{.15in}

First define $a'$ to be the column containing the first element of $\DL{}$.   Lemma \ref{lemma_dec_reverse} implies that $\DL{}\rightarrow b = \ds{a',b}[P(\DL{}\rightarrow b)]$.  Letting $P^* = P(\DL{}\rightarrow b)$, Case 1 gives 

$$\ins\ P^* = \ins[ P^*(a'\leftarrow (\DL{}\rightarrow b))] = \ins\ P.$$
\end{proof}

\noindent\textbf{Acknowledgment.} The first author would like to thank his thesis advisor, Sergi Elizalde, for suggesting the problem of reformulating $\phi^*$.

\vspace{.35in}

\centerline{\textbf{References}}
\vspace{.15in}
\footnotesize
\noindent1. J. Backelin, J. West, G. Xin, Wilf-equivalence for
singleton classes, \emph{Adv. Appl. Math.} \textbf{38} (2007), 133--148.\\[5pt]
2. J. Bloom and D. Saracino, Another look at pattern-avoiding permutations, \emph{Adv. Appl. Math.} \textbf{45} (2010), 395--409.\\[5pt]
3. M. Bousquet-M$\acute{\textrm{e}}$lou, E. Steingr$\acute{\textrm{\i}}$msson, Decreasing
subsequences in permutations and Wilf equivalence for involutions,
\emph{J. Algebraic Combin.} \textbf{22} (2005), 383--409.\\[5pt]
4. S. Fomin, Generalized Robinson-Schensted-Knuth correspondence, \emph{Zapiski Nauchn. Sem. LOMI} \textbf{155} (1986), 156--175.\\[5pt]
5. S. Fomin, Schensted algorithms for dual graded graphs, \emph{J. Alg. Comb.} \textbf{4} (1995), 5--45.\\[5pt]
6. C. Krattenthaler,   Growth diagrams, and increasing and decreasing chains in fillings of Ferrers shapes, \emph{Adv. Appl. Math.} \textbf{37} (2006), 404--431.\\[5pt]
7. B. Sagan, The Symmetric Group, second edition, Springer-Verlag, New York 2001. \\[5pt] 
8. D. Saracino, On two bijections from $S_n(321)$ to $S_n(132)$, \emph{Ars. Comb.} \textbf{101} (2011), 65-74.\\[5pt]
9. C. Schensted, Longest increasing and decreasing subsequences, \emph{Canad. J. Math} \textbf{13} (1961), 179--191.\\[5pt]
10. M. Sch$\ddot{\textrm{u}}$tzenberger, Quelques remarques sur une construction de Schensted, \emph{Math. Scand.} \textbf{12} (1963), 117--128.\\[5pt]
11. R. Stanley,  Enumerative Combinatorics, Vol. II, Cambridge Univ. Press, Cambridge,  1999.

\end{document}